\documentclass[11pt]{article}
\usepackage{amsmath,amsthm,amsfonts,amscd,bezier}
\usepackage{amssymb}
\usepackage{enumerate}
\usepackage[T1]{fontenc}
\usepackage[usenames]{color}

\newtheorem{lemma}{Lemma}[section]

\newtheorem{theorem}[lemma]{Theorem}
\newtheorem{remark}[lemma]{Remark}
\newtheorem{proposition}[lemma]{Proposition}
\newtheorem{corollary}[lemma]{Corollary}
\newtheorem{definition}[lemma]{Definition}
\newtheorem{example}[lemma]{Example}

\usepackage{hyperref}

\hypersetup{colorlinks=true,linkcolor=blue,filecolor=blue,urlcolor=blue,citecolor=blue}

\title{{\bf Toric varieties with isolated singularity and smooth normalization}}

\author{Thaís Maria Dalbelo, Maria Elenice Rodrigues Hernandes and \\  Maria Aparecida Soares Ruas\thanks{T. M. Dalbelo is partially supported by FAPESP grants 2019/21181-0 and 2024/22060-0 and by CNPq grant 403959/2023-3. M. E. Rodrigues Hernandes is partially supported by FAPESP grant 2019/07316-0 and CNPq grant 407454/2023-3, and M. A. S. Ruas is partially supported by CNPq grants 407454/2023-3 and 305695/2019-3 and by FAPESP grant 2019/21181-0.}}

\date{}
\begin{document}
\maketitle

\markboth{T. M. Dalbelo and M. E. Rodrigues
Hernandes and M. A. S. Ruas}{Toric varieties with isolated singularity and smooth normalization}

\begin{center} 2020 Mathematics Subject Classification: 14B05; 32S05; 13F65; 14M25
	
keywords: toric varieties, smooth normalization, semigroup. \end{center}

\noindent {}

\begin{abstract}
In this work, we describe a prenormal form for the generators of the semigroup of a toric variety $X \subset \mathbb{C}^p$ with isolated singularity at the origin and smooth normalization. A complete description of the semigroup is given when $X$ is a variety of dimension $n$ in $\mathbb{C}^{2n}$. Moreover, for toric surfaces in $\mathbb{C}^4$, we provide a set of generators of the ideal $I$ defining $X$.
\end{abstract}

\section{Introduction}

Toric varieties are algebraic varieties defined by prime ideals generated by binomials. Several properties of binomial ideals are explored by Eisenbud and Sturmfels in \cite{ES}, and this theory became an useful tool for studying toric varieties (see \cite{F}, \cite{Cox}, \cite{O}, and \cite{BSR}).

A relevant aspect in the study of the toric varieties is the fact that many difficult questions for general varieties admit concrete solutions in the toric case, using the combinatorics that reside in these varieties. In the affine case, this combinatorics comes from a semigroup  associated with the action of the torus. 

The aim of this work is the study of non normal toric varieties with isolated singularity at the origin whose normalization is smooth, based on their semigroup classification. For this approach we combine tools from the global theory of toric varieties with the local theory of singular map-germs. In Proposition \ref{embedding}, we prove that if $X \subset \mathbb{C}^p$ is an $s$-dimensional variety (not necessarily toric) with isolated singularity admitting smooth normalization $f:\mathbb{C}^s \to X$, then $p\geq 2s$. For $p=2s$, we get a complete classification of the semigroup for this class of toric variety (Theorem \ref{global-2s}). More precisely, an $s$-dimensional toric variety $X(S)\subset \mathbb{C}^{2s}$ has isolated singularity at the origin and smooth normalization if and only if the semigroup $S$ is generated (up to isomorphism) by the set $$\{e_1, \ldots, e_{s-1},\lambda_1e_1+e_s, \ldots,\lambda_2e_2+e_s,\ldots, \lambda_{s-1}e_{s-1}+e_s,ne_s,me_s\},$$ where $\lambda_i,n,m \in \mathbb{N}$, $gcd(n,m)=1$ and $\{e_1, \ldots, e_s\}$ is the canonical basis of $\mathbb{Z}^s$, for $i \in \{1, \ldots, s\}$. For $p>2s$ we obtain a prenormal form for the semigroup $S$ (Theorem \ref{global} for toric surfaces and Theorem \ref{global-s-p} for toric varieties of higher dimension). 

In particular, we prove that a toric surface $X(S)$ in $\mathbb{C}^4$ has isolated singularity and smooth normalization if and only if $S$ is generated by the set $\{(1,0),(\lambda, 1),(0,n),(0,m) \}$ with $\lambda, n,m \in \mathbb{N}$ with $gcd(n,m)=1$. Based on this normal form of the semigroup, we also provide a set of generators of the ideal $I_S$ defining $X(S)$, using the classical process of successive division of $n$ and $m$. Our motivation was the work of Riemenschneider in \cite{Rie2} on the study of normal toric surfaces. It is well known that for affine normal toric varieties, the semigroup can be given as $S = {\check{\sigma}} \cap \mathbb{Z}^s$, where $\check{\sigma} \subset \mathbb{R}^s$ is the dual of a strongly convex cone $\sigma \subset \mathbb{R}^s$. Moreover there exist $p>0$ and $0 \leq q < p$ co-prime integers such that $\sigma$ is generated by the vectors $e_2$ and $pe_1 - qe_2$, where $\{e_1, e_2\}$ are the vectors of the canonical basis of $\mathbb{R}^2$. Many geometrical, topological and algebraic properties of $X(S)$ can be described by the combinatorics coming from this normal form.  Riemenschneider in \cite{Rie2} explicitly calculates a set of minimal generators for the ideal $I_S$ of a normal toric surface $X(S)$, using a process of division by continued fractions of $p/(p-q)$. See also \cite{RGBT, DanielMartinez, ES} for recent results related to this work.

The paper is structured as follows. In Section $2$, we present the basic concepts of toric varieties, their normalization, and singular map-germs. In Section $3$, we give a prenormal form for the semigroup of toric varieties with isolated singularity having smooth normalization, and a normal form for some classes of these varieties. In Section $4$, we use the normal form of the semigroup of a toric surface in $\mathbb{C}^4$ (Theorem \ref{global-2s}) and its combinatorics to compute a generating set of the ideal defining such surfaces. 

\section{Basic notions and results}

For the convenience of the reader and to fix some notations, we first review some general facts.

\subsection{Toric Varieties}

In what follows, we define toric variety and give necessary conditions to obtain the normalization of an affine toric variety. For more details see, for instance, \cite{Cox, F, O}.

Let $S \subset \mathbb{Z}^s$ be a finitely generated semigroup with respect to addition, satisfying $0 \in S$ and $\mathbb{Z}S = \mathbb{Z}^s$.

Denote by $\mathbb{C}[z, z^{-1}]$ the Laurent polynomial ring $\mathbb{C}[z_1, \dots, z_s, z_1^{-1}, \dots, z_s^{-1}]$. We define the support of $f = \sum \lambda_a z^a \in \mathbb{C}[z, z^{-1}]$ by $supp(f) = \{ a \in \mathbb{Z}^s;  \ \ \lambda_a \neq 0\}$. The set
$$
\mathbb{C}[S] = \{f \in \mathbb{C}[z, z^{-1}]; \ \ supp(f) \subset S \}
$$ is an integral domain and finitely generated as a $\mathbb{C}$-algebra.

\begin{definition}
	The $s$-dimensional affine toric variety $X(S)$ is defined by the maximal spectrum of $\mathbb{C}[S]$, that is, $X(S)=\rm{Spec}(\mathbb{C}[S])$.
\end{definition}

The affine toric variety $X(S)$ admits a natural action of the algebraic torus $T = (\mathbb{C}^*)^s$. This action has an open and dense orbit (in the Zariski topology), denoted by $\mathcal{O}$, 
which is homeomorphic to the torus $T$. Denote by $C(S)=\mathbb{R}_{\geq 0} S$ the cone generated by $S$ in ${\mathbb{R}}^s$, and assume that $C(S)$ is a strongly convex polyhedral cone in ${\mathbb{R}}^s$ $(C(S)\cap \{-C(S)\}=\{0\})$ with $dim C(S) = s$. Thus, this action has an unique $0$-dimensional orbit, which is the origin.

Let $\{\gamma_1,\dots,\gamma_p\} \subset \mathbb{Z}^s$ be a set of generators of the semigroup $S$. This set induces a homomorphism of groups $\pi_{S} : \mathbb{Z}^p \to \mathbb{Z}^s$ given by
$$
\pi_{S}(\alpha_1,\dots,\alpha_p) = \alpha_1 \gamma_1 + \dots + \alpha_p \gamma_p.
$$

Given $\alpha = (\alpha_1, \dots, \alpha_p) \in Ker(\pi_{S})$ we set 
$$
{\alpha}_+ = \displaystyle \sum_{\alpha_i > 0} \alpha_i e_i \ \ \ \ \text{and} \ \ \ \ {\alpha}_{-} = -\displaystyle \sum_{\alpha_i < 0} \alpha_i e_i,
$$
where $e_1, \dots, e_p$ are the elements of the standard basis of $\mathbb{R}^p$. Let us observe that $\alpha = \alpha_{+} - \alpha_{-}$ and $\alpha_{+}, \alpha_{-} \in \mathbb{N}^p$.

Consider the ideal 
\begin{equation}\label{nucleo}
	I_S = \langle x^{\alpha_{+}} - x^{\alpha_{-}}; \ \ \alpha \in Ker(\pi_{S}) \rangle \subset \mathbb{C}[x_1,\dots,x_p],
\end{equation}
where $x^{\beta} = x_1^{\beta_1} \dots x_p^{\beta_p}$ with $\beta = (\beta_1, \dots, \beta_p) \in \mathbb{N}^p$. The ideal $I_S$ is a prime ideal generated by binomials (see \cite[Proposition $1.1.11$]{Cox}). Moreover, $X(S) $ is the irreducible affine variety $\mathcal{V}(I_S)$ defined by the zero set of $I_S$, which is not necessarily
normal. 

\begin{definition}
	A semigroup $S \subset \mathbb{Z}^s $ is saturated if for all $k \in \mathbb{N}\setminus\{0\}$ and
	$m \in \mathbb{Z}^s$, $km \in S$ implies $m \in S$.
\end{definition}

For instance, if $\sigma \subset \mathbb{R}^s$ is a strongly convex rational polyhedral cone, then the semigroup $S_{\sigma} = \check{\sigma} \cap \mathbb{Z}^s$ is saturated, where $\check{\sigma}$ is the dual cone of $\sigma$.

In the next theorem, we can see the relation between normal toric varieties and strongly convex rational polyhedral cones. A proof of this result can be found in \cite[Theorem $1.3.5$]{Cox}.

\begin{theorem}
	Let $X(S)$ be an affine toric variety. Then the following are equivalent:
	\begin{itemize}
		\item[1)] $X(S)$ is normal;
		
		\item[2)] $S \subset \mathbb{Z}^s$ is a saturated semigroup;
		
		\item[3)] $X(S) = \rm{Spec}(\mathbb{C}[S])$, where $S= \check{\sigma} \cap \mathbb{Z}^s$ and $ \sigma \subset \mathbb{R}^s$ is a strongly convex
		rational polyhedral cone.
	\end{itemize}
\end{theorem}

In order to describe the normalization of an affine toric variety $X(S)$, let $\check{\sigma} = C(S) \subset {\mathbb{R}}^s$ be the convex cone of any
finite generating set of $S$ and denote by $\sigma = \check{C}(S) \subset {\mathbb{R}}^s$ its dual cone. For a proof of the following proposition see, for instance, \cite[Proposition 1.3.8]{Cox}.

\begin{proposition}\label{normalization}
	With the above notations, the cone $\sigma$ is a strongly convex rational polyhedral cone
	in ${\mathbb{R}}^s$ and the inclusion $\mathbb{C}[S] \subset \mathbb{C}[\check{\sigma} \cap \mathbb{Z}^s]$ induces a morphism $ f: X(\check{\sigma} \cap \mathbb{Z}^s) \to X(S)$ which is the
	normalization map of $X(S)$.
\end{proposition}

As a direct consequence of this result, the normalization $f$ is a monomial map (see \cite[Example $1.3.9$]{Cox}).

Moreover, the restriction $f: X(\check{\sigma} \cap \mathbb{Z}^s) \setminus f^{-1}(\Sigma(X(S))) \to X(S) \setminus \Sigma(X(S))$ is an isomorphism, where $\Sigma(X(S))$ denotes the singular set of $X(S)$.

\subsection{Map-germs singularities} In Singularity Theory, an important class of map-germs $f:(\mathbb{C}^s,0) \to (\mathbb{C}^p,0)$ is the so-called finitely determined with respect to some group acting on the space of these map-germs.

Denote by $\mathcal{R}$ and $\mathcal{L}$ the groups of germs of analytic diffeomorphisms of $\mathbb{C}^s$ and $\mathbb{C}^p$ at the origin, respectively. The right-left equivalence group is defined as $$\mathcal{A}=\{(\rho,\psi); \ \rho \in \mathcal{R}\ \ \mbox{and}\ \ \psi \in \mathcal{L}\}.$$ We say that two map-germs $f$ and $g$ from $(\mathbb{C}^s,0)$ to $(\mathbb{C}^p,0)$ are $\mathcal{A}$-equivalent if there exists $(\rho,\psi) \in  \mathcal{A}$ such that $g=\psi \circ f \circ \rho^{-1}$.
\begin{definition}
	A map-germ $f:(\mathbb{C}^s,0)\to (\mathbb{C}^p,0)$ is $k$-$\mathcal{A}$-determined if for every $g:(\mathbb{C}^s,0)\to (\mathbb{C}^p,0)$ with the same $k$-jet of $f$, denoted by $j^kg(0)=j^kf(0)$, it follows that $f$ is $\mathcal{A}$-equivalent to $g$. We say that $f$ is $\mathcal{A}$-finitely determined or $\mathcal{A}$-finite if $f$ is $k$-$\mathcal{A}$-determined for some $k< \infty$.
\end{definition}

If $f:(\mathbb{C}^s,0)\to (\mathbb{C}^p,0)$ is a map-germ and $p \geq 2s$, then it follows from Mather-Gaffney geometric criterion of finite determinacy (see \cite{Wall}) that $f$ is $\mathcal{A}$-finite if and only if there exist a neighbourhood $U \subset \mathbb{C}^s$ of the origin and a representative $f: U  \to \mathbb{C}^p$ such that $f$ has isolated singularity at the origin and $f|_{U \setminus \{0\}}$ is an embedding. As a consequence the variety $X=f(U)\subset \mathbb{C}^p$ has an isolated singularity at the origin. 

Since the normalization map of a toric variety $X(S)$ is a monomial map, we relate monomial $\mathcal{A}$-finite map-germs obtained in \cite{EleniceCidinha1} and \cite{EleniceCidinha2} to properties of toric varieties with isolated singularity and smooth normalization map.

The proofs of the results in the next section are based on the general form of an $\mathcal{A}$-finite monomial map-germ from $(\mathbb{C}^s,0)$ to $(\mathbb{C}^p,0)$ for $p\geq 2s$, given in \cite[Theorem $4.4$]{EleniceCidinha2}.
When $p=2s$, Corollary $4.3$ in \cite{EleniceCidinha2} gives a complete classification result. More precisely, any singular monomial map-germ from $(\mathbb{C}^s,0)$ to $(\mathbb{C}^{2s},0)$ is $\mathcal{A}$-finite if and only if it is $\mathcal{A}$-equivalent to $$f(x_1,\ldots, x_{s-1},y)=(x_1, \ldots, x_{s-1},y^{n},y^{m}, x_1^{\lambda_1}y, \ldots, x_{s-1}^{\lambda_{s-1}}y),$$ for $\lambda_i,n,m\in \mathbb{N}$, with $i\in \{1,\ldots, s-1\}$, $1<n < m$, and $gcd(n,m)=1$. However, when $p > 2s$, the operations defined in \cite{EleniceCidinha2} produce, in general, families of $\mathcal{A}$-finite monomial singularities. Some monomial components of the prenormal form of the germ are essential to produce an $\mathcal{A}$-finite singularity, but we have freedom to choose monomials for the other components. In the next example, we illustrate this fact.
\begin{example}
	\label{Ex-Surf-C5}
	By \cite[Theorem $4.4$]{EleniceCidinha2}, a singular $\mathcal{A}$-finite monomial map-germ from $(\mathbb{C}^2,0)$ to $(\mathbb{C}^5,0)$ is $\mathcal{A}$-equivalent to one of the following map-germs:
	\begin{enumerate}
		\item $(x,y^{m_1},y^{m_2},x^{\lambda}y,0)$
		\item $(x,y^{m_1},y^{m_2},x^{\lambda}y,x^ay^b)$, with $x^ay^b \notin \mathbb{C}\{x,y^{m_1},y^{m_2},x^{\lambda}y\}$
		\item $(x,y^n,y^m,y^l,x^{\lambda}y)$
		\item $(x^{n_1},x^{n_2},y^{m_1},y^{m_2},xy)$     
	\end{enumerate}
	with $1<m_1<m_2$, $1< n<m<l$, and $1<n_1<n_2$ satisfying $gcd(n,m,l)=gcd(n_1,n_2)=gcd(m_1,m_2)=1$, $\lambda \geq 1$, and $a,b \neq 0$.
\end{example}

\section{Semigroup of isolated toric singularities with smooth normalization}

In this section, we give a prenormal form for the semigroup of an $s$-dimensional toric variety $X(S) \subset \mathbb{C}^p,$  with isolated singularity at the origin and smooth normalization (Theorems \ref{global} and \ref{global-s-p}). With these hypotheses, we prove that $p \geq 2s$ (Proposition \ref{embedding}), and when $p=2s$ we give a complete classification for $S$ (Theorem \ref{global-2s}).

The next result holds in general for any  $s$-dimensional analytic variety $X \subset \mathbb C^p, $ with isolated singularity at the origin and whose normalization map is
$f: \mathbb C^s \to X.$

\begin{proposition}\label{embedding} Let $X \subset \mathbb C^p $ be an $s$-dimensional  variety with isolated singularity at the origin and admitting smooth normalization $f: \mathbb C^s \to X.$ Then, $p \geq 2s.$ 
\end{proposition}
\begin{proof}
	Let  $f: \mathbb{C}^s \to X$ be the normalization map and $\phi : X \to \mathbb C^p$ an embedding.
	Thus, the composite $\tilde{f}=\phi \circ f : \mathbb C^s \to \mathbb C^p$ is also finite and generically one-to-one. Since $X$ has isolated singularity, these properties imply respectively that $\tilde{f}$ has isolated singularity at the origin and its restriction to $U \setminus \{0\}$ is an embedding, where $U$ is a neighbourhood of the origin. By Mather-Gaffney geometric criterion of finite determinacy, the germ of $\tilde{f}$ at the origin is $\mathcal{L}$-finitely
	determined (hence also $\mathcal A$-finitely determined). Then, it follows from Theorem 2.5 in \cite{Wall} that  $p \geq 2s$. This conclusion does not depend on the maps $f$ and $\phi$, since the normalization is unique.
\end{proof}

Let $S \subset \mathbb{Z}^s$ be the semigroup defining $X(S)$. As we said before, $C(S) = \mathbb{R}_{\geq 0} S$ is a strongly convex cone. This fact and the hypothesis of smooth normalization allow us to assume, up to a change of coordinates, that $S \subset \mathbb{N}^s$ and
$C(S)\cap \mathbb{Z}^s = \mathbb{N}^s$ and therefore, in this case, the normalization map is given by $f: \mathbb{C}^s \to X(S)$. We use the same notation $\tilde{f}:(\mathbb{C}^s,0) \to (\mathbb{C}^p,0)$ for the germ at the origin of $\tilde{f}$, as in the proof of Proposition \ref{embedding}.

\begin{proposition}\label{Afinitamentedeterminada} Let $X(S) \subset \mathbb{C}^p$ be an $s$-dimensional toric variety with isolated singularity at the origin and $p \geq 2s$. If $f: \mathbb{C}^s \to X(S)$ is the normalization map of $X(S),$ then  the following conditions hold:
	
	\begin{itemize} 
		\item [(1)] $\tilde{f} :  (\mathbb{C}^s,0) \to (\mathbb{C}^p,0)$ is an $\mathcal{A}$-finite monomial map-germ. 
		\item [(2)]	 The singular set of $X(S),$ denoted $\Sigma X(S),$ reduces to $0.$
		\item [(3)]  $f^{-1}(0)=\{0\}$.
	\end{itemize}	
\end{proposition}
\begin{proof}
	We proved in Proposition \ref{embedding} that the map-germ $\tilde{f}$ is $\mathcal A$-finite and, as a consequence of Proposition \ref{normalization}, the normalization map $f$ is monomial. Hence, (1) holds. 
	
	To prove (2), recall that the action from $(\mathbb{C}^*)^s$ on $X(S)$ has an unique $0$-dimensional orbit, which is the origin of $\mathbb{C}^p$. As $X(S)$ has an isolated singularity, then $0$ is actually its only singularity. 
	
	Now, (1) and (2) imply (3). In fact, suppose that  $f^{-1}(0) \varsupsetneq \{0\}.$ As $f$ is a monomial map,
	it follows that $f^{-1}(0)$ is a union of linear spaces  passing through the origin, which contradicts the 
	$\mathcal A$-finiteness of the germ $\tilde{f}$ at the origin. For another proof, see for instance \cite[Lemma 2.3]{ArturoDaniel}.
\end{proof}

\begin{proposition}
	\label{Proposition-Isom-Semig}
	Let $f,g:(\mathbb{C}^s,0) \to (\mathbb{C}^p,0)$ be $\mathcal{A}$-finite monomial map-germs with $s < p$. Denote by $S_f$ and $S_g$ the semigroups generated by the exponents of the components of $f$ and $g$, respectively. Then, $f \underset{\mathcal{A}}{\sim}g$ if and only if $S_f$ and $S_g$ are isomorphic. 
\end{proposition}
\begin{proof} First of all, we observe that if $f$ is $\mathcal{A}$-finite, $s < p$, and $(f(\mathbb{C}^s),0)$ is the germ of the image of $f$ at the origin, then $f:(\mathbb{C}^s,0) \to (f(\mathbb{C}^s),0)$ is the normalization map (see \cite[Proposition $D.4$]{BallesterosMond}). Therefore, $(f(\mathbb{C}^s),0)=(X(S_f),0)$. 
	
	If the semigroups $S_f$ and $S_g$ are isomorphic, then the toric varieties $X(S_f)$ and $X(S_g)$ are isomorphic. Consequently, so are their germs $(X(S_f),0)$ and $(X(S_g),0)$. Therefore, the map-germs $f$ and $g$ are ${\mathcal{A}}$-equivalent (see \cite[Corollary 5.11]{ThesisGaffney}).
	
	If $f$ and $g$ are ${\mathcal{A}}$-equivalent, then by \cite[Corollary 5.10]{ThesisGaffney} the $\mathbb{C}$-algebras $f^*(\mathcal{O}_p) = \mathbb{C}\{f_1,\dots,f_p\}$ and $g^*(\mathcal{O}_p) = \mathbb{C}\{g_1,\dots,g_p\}$ are isomorphic, with $f_i$ and $g_i$ monic monomials, for $i\in \{1,\ldots , p\}$. Thus, there exists an isomorphism between the multiplicative groups of monic monomials of $\mathbb{C}\{f_1,\dots,f_p\}$ and $\mathbb{C}\{g_1,\dots,g_p\}$. Moreover, this one induces an isomorphism from $S_f$ to $S_g$. 
\end{proof}

\vspace{0.2cm} The next theorem provides a prenormal form for the semigroup of toric surfaces with isolated singularity and smooth normalization.

\begin{theorem}\label{global}
	Let $X(S) \subset \mathbb{C}^p$ be a toric surface. Then, $X(S)$ has isolated singularity at the origin and  normalization map $f: \mathbb{C}^2 \to X(S)$ if and only if $p \geq 4$ and $S$ is generated (up to isomorphism) by one of the following sets:
	\begin{itemize}
		\item[1.] $ \{(1,0), (0,m_1), \dots, (0,m_l), (\lambda, 1), (a_1,b_1), \dots, (a_{p-2-l}, b_{p-2-l})\}$, for some integer $l$ with $2\leq l \leq p-2$ or 
		
		\item[2.] $
		\{(n_1,0), \dots, (n_k,0), (0,m_1), \dots, (0,m_l), (\lambda, 1), (1,b_1), (a_2,b_2), \dots, (a_{t}, b_{t}),\ \}$ for some $k, l \geq 2$ with $k +l \leq p-1$ and  $t= p-1-(k+l)$,
	\end{itemize}
	where $gcd(m_1, \dots,m_l)=1 = gcd (n_1, \dots, n_k)$, $a_i, b_i, \lambda \in \mathbb{N}$ with $n_k > \dots > n_1 > 1$ and $m_l > \dots > m_1 > 1$. Moreover, if $k+l=p-1$, then $\lambda=1$.
\end{theorem}
\begin{proof} 
	Let $X(S)$ be a toric surface with isolated singularity at the origin and normalization map $f: \mathbb{C}^2 \to X(S)$. By Proposition \ref{embedding}, it follows that $p \geq 4$. Denote by $\tilde{f}: (\mathbb{C}^2,0) \to (\mathbb{C}^p,0)$ the map-germ of $f$ at the origin and consider $V$ an open neighbourhood of the origin in $\mathbb{C}^2$ satisfying $$\tilde{f}(\mathbb{C}^2 \cap V) = f(\mathbb{C}^2 \cap V).$$
	
	By Proposition \ref{Afinitamentedeterminada}, the map-germ $\tilde{f}$ is monomial and $\mathcal{A}$-finite. If $\tilde{f}$ has corank $1$, by \cite[Theorem $4.4$]{EleniceCidinha2} there exists $l$ with $2\leq l \leq p-2$ such that $\tilde{f}|_{V}$ is $\mathcal{A}$-equivalent to
	\begin{equation}
		\label{Map-g-Corank-1}
		g(x,y) = (x, y^{m_1}, \dots, y^{m_l}, x^{\lambda}y, x^{a_1} y^{b_1}, \dots, x^{a_{p-l-2}} y^{b_{p-l-2}})   
	\end{equation}
	for some $m_i,a_j, b_j, \lambda \in \mathbb{N}$, where $gcd(m_1, \dots,m_l)=1$, and $m_l > \dots > m_1 > 1$ for $i\in \{1,\ldots, l\}$ and $j\in \{1,\ldots, p-l-2\}$. 
	
	Consider $S_g$ the semigroup generated by the exponents of the components of $g$ and $X(S_g) \subset \mathbb{C}^p$ the toric surface generated by $S_g$. By Proposition \ref{Proposition-Isom-Semig}, $S_g$ and $S_{\tilde{f}|_{V}}$ are isomorphic. Since $(1,0) \in S_g$ and $gcd(m_1, \dots,m_l)=1$, we have $\mathbb{Z}{S_g} = \mathbb{Z}^2$. Moreover the cone $\check{\sigma} =
	C(S_g)$ is the first quadrant of $\mathbb{R}^2$. It follows from Proposition \ref{normalization} that the normalization map $F: \mathbb{C}^2 \to X(S_g)$ of $X(S_g)$ is given by
	$$F(x,y) = (x, y^{m_1}, \dots, y^{m_l}, x^{\lambda}y, x^{a_1} y^{b_1}, \dots, x^{a_{p-l-2}} y^{b_{p-l-2}}).$$ 
	Moreover, $F|_V =g \underset{\mathcal{A}}{\sim}\tilde{f}|_V= f|_V$. As consequence of Proposition \ref{normalization}, $f((\mathbb{C}^*)^2)$ and $F((\mathbb{C}^*)^2)$ are the dense orbits of the torus action in $X(S)$ and $X(S_g)$, respectively. Therefore, if they coincide in $V \cap (\mathbb{C}^*)^2$, they are indeed equal. Then $X(S) = X(S_g)$ and $f = F$.
	
	On the other hand, if $\tilde{f}:(\mathbb{C}^2,0) \to (\mathbb{C}^p,0)$ has corank $2$, by \cite[Theorem $4.4$]{EleniceCidinha2} 
	there exist $k, l \geq 2$ with $ k +l \leq p-1$ such that $\tilde{f}|_{V}$ is $\mathcal{A}$-equivalent to
	$$
	g(x,y)=(x^{n_1}, \dots, x^{n_k},y^{m_1}, \dots, y^{m_l}, x^{\lambda}  y,
	x  y^{b_1}, x^{a_2}  y^{b_2}, \dots, x^{a_{p-1-(k+l)}} y^{ b_{p-1-(k+l)}}),
	$$
	where $gcd(n_1, \dots, n_k)=1=gcd(m_1, \dots,m_l)$, $a_i, b_i, \lambda \in \mathbb{N}$,  $n_k > \dots > n_1 > 1$ and $m_l > \dots > m_1 > 1$, in which if $p-1=k+l$, then $\lambda=b_1$. Proceeding exactly in the same way as in the previous item, we can prove that $F = f$.
	
	Now suppose that the semigroup $S$ is given as in item $1$  (the argument for the semigroup of item $2$ is analogous). Then $\mathbb{Z} S = \mathbb{Z}^2$ and, by Proposition \ref{normalization}, the monomial map $f: \mathbb{C}^2 \to \mathbb{C}^p$ as in (\ref{Map-g-Corank-1}) is the normalization map of $X(S)=f(\mathbb{C}^2)$. Thus the map-germ of $f$ at the origin $\tilde{f}: (\mathbb{C}^2,0) \to (\mathbb{C}^p,0)$ is $\mathcal{A}$-finite (\cite[Theorem $4.4$]{EleniceCidinha2}). By Mather-Gaffney geometric criterion of finite determinacy, there exists a neighborhood $U$ of $0$ such that it is an isolated singularity of $\tilde{f}$ and $\tilde{f}|_{U \setminus \{0\}}$ is an embedding. Using the torus action as above, we conclude that $X(S)$ has an isolated singularity at the origin.\end{proof}	

\begin{example}
	With the same notations as in Example \ref{Ex-Surf-C5} and by previous theorem, a toric surface $X(S)$ in $\mathbb{C}^5$ has isolated singularity and smooth normalization map if and only if the semigroup $S$ is isomorphic to the one generated by the sets:
	\begin{enumerate}
		\item $\{(1,0),(0,m_1),(0,m_2),(\lambda, 1),(a,b)\}$
		\item $\{(1,0),(0,n),(0,m),(0,l),(\lambda, 1)\}$
		\item $\{(n_1,0),(n_2,0),(0,m_1),(0,m_2),(1, 1)\}$, with $a,b,\lambda \in \mathbb{N}$.
	\end{enumerate}
\end{example}

The next results generalize Theorem \ref{global}. More precisely, we give a prenormal form for the generators of the semigroup $S$ of a toric variety with isolated singularity and smooth normalization. In some cases, according to the dimension of the variety and its embedding dimension, a normal form for the generators of $S$ is obtained, as in the following theorem.

\begin{theorem}\label{global-2s}
	Let $X(S) \subset \mathbb{C}^{2s}$ be an $s$-dimensional toric variety. Then, $X(S)$ has isolated singularity at the origin and  normalization map $f: \mathbb{C}^s \to X(S)$ if and only if the semigroup $S\subset \mathbb{Z}^s$ is generated (up to isomorphism) by the set 
	\begin{equation*} 
		\begin{array}{c}
			\{(1,0, \dots,0), \dots, (0, \dots,1,0), (\lambda_1, 0, \dots, 0, 1), \dots,
			(0, \lambda_2, 0, \dots, 0, 1), \\ 
			(0, \dots, \lambda_{s-2},0, 1), (0, \dots, 0, \lambda_{s-1}, 1), (0, \dots, 0, n), (0, \dots, 0, m) \}
		\end{array}
	\end{equation*}
	where $\lambda_i,n,m \in \mathbb{N}$ with $1 < n< m$, and $gcd(n,m)=1$ for $i=1, \dots, s-1$. For $s=2$, the semigroup $S$ is generated by $\{(1,0),(\lambda,1),(0,n),(0,m)\}$ with $\lambda \in \mathbb{N}$.
\end{theorem}
\begin{proof}
	The proof is analogous to Theorem \ref{global}, applying \cite[Corollary $4.3$]{EleniceCidinha2}. For toric surfaces in $\mathbb{C}^4$, see also \cite[Theorem $2.4$]{EleniceCidinha1}.
\end{proof}

We denote by $\{e_1,\ldots, e_{s-q}, e_{s-q+1}, \ldots, e_s\}$ the canonical basis in $\mathbb{Z}^s$ for some $q \in \{1,\ldots, s\}$.

\begin{proposition}
	\label{Semigroup-Sq}
	For a fixed $q$, with $q \in \{1, \ldots, s\}$, let $S_q \subset \mathbb{Z}^s$ be the semigroup generated by the vectors:
	\begin{enumerate}
		\item $e_1, \ldots, e_{s-q}\ $;
		\item $m_{jl}\cdot e_{s-q+j}, \ $ for $j\in \{1, \ldots, q\}$, $l\in \{1, \ldots, k_j\}$, $gcd(m_{j1},\ldots, m_{jk_j})=1$ for some $k_j \geq 2$ and $m_{jl} >1$;
		\item $\lambda_{ij} \cdot e_i+e_{s-q+j},\ $ for $i\in \{1, \ldots, s-q\}$, $j\in \{1, \ldots, q\}$, and $\lambda_{ij}\geq 1$;
		\item $\mu_{jt} \cdot e_{s-q+j}+e_{s-q+t},\ $ for $j,t\in \{1, \ldots, q\}$ with $t\neq j$, and $\mu_{jt}\geq 1$.
	\end{enumerate}
	Then, $X(S_q)$ is an $s$-dimensional toric variety in $\mathbb{C}^r$ with isolated singularity and smooth normalization, where $r$ is the sum of the number of elements in $(1),(2),(3)$ and $(4)$. Moreover, $r \geq 2s$ and the embedded dimension $r_q$ of $X(S_q)$ satisfy 
	$$s(q+1) - \frac{q(q-1)}{2} \leq r_q \leq r=s(q+1)-2q+\sum_{j=1}^qk_j,$$ where the lower bound occurs when $k_j=2$ and $\mu_{jt}=1$ for all $j,t \in \{1, \ldots, q\}$ with $t\neq j$.
\end{proposition}
\begin{proof}
	It is easy to see that $\mathbb{Z} S_q = \mathbb{Z}^s$ and
	$C(S_q)\cap \mathbb{Z}^s = \mathbb{N}^s$. By Proposition \ref{normalization}, the monomial map $f: \mathbb{C}^s \to \mathbb{C}^p$, whose exponents of its components are the generators of $S_q$, is the normalization map of $X(S_q)=f(\mathbb{C}^s)$. By \cite[Theorem $4.4$]{EleniceCidinha2} the map-germ of $f$ at the origin $\tilde{f}: (\mathbb{C}^s,0) \to (\mathbb{C}^p,0)$ is $\mathcal{A}$-finite. As in the proof of Theorem \ref{global}, we apply the Mather-Gaffney geometric criterion and properties of torus action to conclude that $X(S_q)$ is an $s$-dimensional toric variety in $\mathbb{C}^r$ with isolated singularity and smooth normalization.
	
	In \cite[Corollary 4.6]{EleniceCidinha2} we prove that $r_q \geq s(q+1) - \frac{q(q-1)}{2}$, where equality holds when $\tilde{f}(x_1,\ldots,x_{s-q},y_1,\ldots,y_q)$ is $\mathcal{A}$-equivalent to $$(I,y_1^{m_{11}},y_1^{m_{12}}, \ldots,y_q^{m_{q1}},y_q^{m_{q2}},x_1^{\lambda_1}y_1,\ldots, x_{s-q}^{\lambda_{s-q}}y_q,y_1y_2,\ldots,y_1y_q,y_2y_3,\ldots,y_2y_q,\ldots,y_{q-1}y_q),$$ where $I$ is the identity map-germ on $(\mathbb{C}^{s-q},0)$. The upper bound of $r_q$ is the cardinality of $S_q$ (see also the general form of $\tilde{f}$ in \cite[Theorem $4.4$]{EleniceCidinha2}).
\end{proof}

\begin{example}
	By the previous result, the $3$-dimensional toric variety in $\mathbb{C}^{10}$, whose semigroup $S_2$ is generated by $$\{e_1, m_{11}e_2,m_{12}e_2, m_{21}e_3,m_{22}e_3,m_{23}e_3,\lambda_{11}e_1+e_2,\lambda_{12}e_1+e_3,\mu_{12}e_2+e_3,\mu_{21}e_3+e_2\}$$ has isolated singularity and smooth normalization map of corank $2$, where $m_{jl}>1$ for $j\in \{1,2\}$ and $l\in \{1,\ldots, k_j\}$ with $k_1=2$ and $k_2=3$, $gcd(m_{11},m_{12})=gcd(m_{21},m_{22},m_{23})=1$, and $\lambda_{11},\lambda_{12},\mu_{12},\mu_{21} \in \mathbb{N}$. If $\mu_{12}=\mu_{21}=1$, then the embedded dimension of $X(S_2)$ is $9$.
\end{example}

\begin{theorem}
	\label{global-s-p}
	Let $X(S)$ be an $s$-dimensional toric variety in $\mathbb{C}^{p}$. Then $X(S)$ has isolated singularity at the origin and smooth normalization map $f: \mathbb{C}^s \to X(S)$ of corank $q$ if and only if $p\geq 2s$ and there exists a semigroup $S_q$ as in Proposition \ref{Semigroup-Sq} such that $S$ is, up to isomorphism, a disjoint union $$S=S_q \cup S_t,$$ where $S_t=\{(a_{i_1},\ldots, a_{i_s}); \ a_{i_j} \in \mathbb{N}^*, \ \mbox{for} \ 1 \leq i \leq t,\ 1 \leq j \leq s,\ \mbox{with}\ p-t=\sharp S_q\}$ or empty.
\end{theorem}
\begin{proof}
	The proof is a generalization of Theorem \ref{global}, using the general form of an $\mathcal{A}$-finite monomial map-germ from $(\mathbb{C}^s,0)$ to $(\mathbb{C}^p,0)$ of corank $q$ for $p \geq 2s$ in \cite[Theorem $4.4$]{EleniceCidinha2}.   
\end{proof}

\section{Toric surfaces in $\mathbb{C}^4$}

Given a class of varieties, describing the equations which define them is, in general, a difficult problem. For normal toric surfaces, Riemenschneider \cite{Rie2} computed
a set of minimal generators for the ideal $I_S$ of the surface $X(S)$, using a process of division by continued fractions. Motivated by this result, we provide a set of generators for the ideal $I_S$ defining $X(S)$, in the case of a family of non-normal toric surfaces in $\mathbb{C}^4$.

In Theorem \ref{global-2s}, we prove that a toric surface $X(S)=\mathcal{V}(I_S)$ in $\mathbb{C}^4$ has isolated singularity and smooth normalization if and only if the semigroup $S$ is generated by $\{(1,0),(\lambda,1),(0,n),(0,m)\}$ with $\lambda,n,m\in \mathbb{N}$, $1<n<m$ and $gcd(n,m)=1$. In the present section, we use this normal form and its combinatorics to explicitly compute a finite set of irreducible binomials generating $I_S$, using the classical process of successive division of $n$ and $m$. In what follows, we denote the coordinates of $\mathbb{C}^4$ by $X,Y,Z$ and $W$.

\begin{proposition}
	\label{Eq-Toric-C4} Let $X(S)=\mathcal{V}(I_S) \subset \mathbb{C}^4$ be a toric surface generated by the semigroup $S = \left\langle (1,0), (\lambda,1), (0,n),(0,m) \right\rangle$ with $\lambda, n,m\in \mathbb{N}$, $1< n< m$ and $gcd(n,m)=1$. Then, $I_S$ is generated by binomials of type
	\begin{equation}
		\label{Eq-Type1}
		Y^aW^b-X^{\lambda a}Z^d,
	\end{equation}
	\begin{equation}
		\label{Eq-Type2}
		Y^aZ^d-X^{\lambda a}W^b
	\end{equation} where $a,b\in \mathbb{N}\cup \{0\}$, $d\in \mathbb{N}$ satisfy $n \cdot d + l \cdot a=m \cdot b$ for $l=\pm 1$, in which $l=-1$ and $l=1$ are associated with the binomials (\ref{Eq-Type1}) and (\ref{Eq-Type2}), respectively.
\end{proposition} 
\begin{proof}
	Let $x^{\alpha_{+}} - x^{\alpha_{-}}$ be a binomial in $I_S$ as in (\ref{nucleo}), where $\alpha=\alpha_{+} - \alpha_{-}=(\alpha_1,\alpha_2,\alpha_3,\alpha_4)$ belongs to $Ker(\pi_S)$ with $\alpha_{+},\alpha_{-} \in \mathbb{N}^4$. Then $\alpha_1(1,0)+\alpha_2(\lambda,1)+\alpha_3(0,n)+\alpha_4(0,m)=(0,0)$, that is, \begin{equation}\label{Syst-Diop}
		\left\{\begin{array}{l}
		\alpha_1+\lambda \alpha_2=0  \\
		\alpha_2+n\alpha_3+m\alpha_4=0.  
	\end{array}\right.\end{equation} Thus, $\alpha_1=-\lambda \alpha_2$. If $\alpha_2 \geq 0$, then $\alpha_1 \leq 0$, and vice versa. Without loss of generality, we consider $\alpha_2 \geq 0$. Note that $\alpha_2=0$ if and only if $\alpha_1=0$, and in this case, $n\alpha_3+m\alpha_4=0$. So $\alpha_3$ and $\alpha_4$ must have opposite signs. Consequently, if $\alpha_4 <0$ we get the binomial $Z^{\alpha_3}-W^{-\alpha_4}$ in $I_S$ (see also \cite[Lemma $2.2$]{DanielMartinez}). Similarly, $W^{\alpha_4}-Z^{-\alpha_3}\in I_S$ if $\alpha_3 <0$. 
	
	If $\alpha_1 \cdot \alpha_2 \neq 0$, the relation $n\alpha_3+m\alpha_4=-\alpha_2$ holds. Thus, $\alpha_3$ and $\alpha_4$ are not simultaneously positive, which implies that $Y^{\alpha_2}Z^{\alpha_3}W^{\alpha_4}-X^{\lambda \alpha_2}$ is not a generator of $I_S$. If $\alpha_3$ and $\alpha_4$ have opposite signs, then the binomials $Y^{\alpha_2}W^{\alpha_4}-X^{\lambda \alpha_2}Z^{-\alpha_3}\in I_S$ if $\alpha_3 <0$ and $Y^{\alpha_2}Z^{\alpha_3}-X^{\lambda \alpha_2}W^{-\alpha_4} \in I_S$ if $\alpha_4<0$. In particular, if we write $m=h_0n + r_1$ with $h_0,r_1 \in \mathbb{N}$ and $0 < r_1< n$ (since $1< n < m$), the integers $\alpha_2=r_1,\alpha_3=h_0$ and $\alpha_4=-1$ satisfy the equation $\alpha_2+n\alpha_3=-m\alpha_4$. Thus, the binomial $Y^{r_1}Z^{h_0}-X^{\lambda r_1}W$ is a generator of $I_S$. If $\alpha_3=0$, from the relation $\alpha_2=-m\alpha_4=-(h_0n+r_1)\alpha_4$, we obtain  $\alpha_4 <0$, and $Y^{\alpha_2}-X^{\lambda \alpha_2}W^{-\alpha_4} \in I_S$. However, this binomial is not irreducible, since  
	$$Y^{-\alpha_4h_0n} Y^{-\alpha_4r_1} - X^{-\lambda \alpha_4 h_0n}X^{-\lambda\alpha_4r_1}W^{-\alpha_4}= Y^{-\alpha_4r_1}(Y^{-\alpha_4h_0n}  - X^{-\lambda \alpha_4 h_0n}Z^{-\alpha_4h_0}),$$ where the last equality follows from  $Y^{r_1}Z^{h_0}-X^{\lambda r_1}W=0$ in $\mathcal{V}(I_S)$. Thus, it is redundant in the generators set of $I_S$. If $\alpha_4=0$, then $n\alpha_3=-\alpha_2$ which implies that $\alpha_3 <0$. Therefore, $Y^{\alpha_2}-X^{\lambda\alpha_2}Z^{-\alpha_3} \in I_S$. In particular, for $\alpha_2=n$ and $\alpha_3=-1$ we get $Y^{n}-X^{\lambda n}Z \in I_S$ (see \cite[Lemma $2.2$]{DanielMartinez}). If $\alpha_3,\alpha_4 <0$, from the relation $\alpha_2=-n\alpha_3-m\alpha_4$ we have the binomial $Y^{\alpha_2}-X^{\lambda \alpha_2}Z^{-\alpha_3}W^{-\alpha_4}$, which is not an element of a minimal set of generators of $I_S$, since 
	$$Y^{-n\alpha_3}Y^{-m\alpha_4}-(X^{-\lambda n \alpha_3}Z^{-\alpha_3})(X^{-\lambda m \alpha_4}W^{-\alpha_4})=Y^{-n\alpha_3}(Y^{-m\alpha_4}-X^{-\lambda m \alpha_4}W^{-\alpha_4}),$$ where the last equality follows from the equation $Y^{n}-X^{\lambda n}Z=0$.
\end{proof}

It follows from Proposition \ref{Eq-Toric-C4}, that each generator of $I_S$ corresponds to an equation of the form $$n \cdot d + l \cdot a = m \cdot b,$$ with $l=\pm 1$, where $l=-1$ and $l=1$ are associated with the binomials (\ref{Eq-Type1}) and (\ref{Eq-Type2}), respectively. Reciprocally, for each $d,a,b \in \mathbb{N}$ satisfying the previous equation, we can associate a binomial in $I_S$ of the type (\ref{Eq-Type1}) if $l=-1$ or (\ref{Eq-Type2}) if $l=1$. 

For each fixed $b$ we have a Diophantine equation\footnote{A Diophantine equation in the variables $x$ and $y$ given by $n\cdot x + l \cdot y = m,$ with $n,l,m \in \mathbb{Z}$, admits solution in $\mathbb{Z}$ if and only if $gcd(n,l)$ divides $m$. In particular, if $gcd(n,l)=1$. Furthermore, if $x_0,y_0 \in \mathbb{Z}$ is a particular solution of the previous equation, then $x=x_0-l\cdot t$ and $y=y_0+n \cdot t$, for all $t \in \mathbb{Z}$, are also solutions in $\mathbb{Z}$.} in the variables $d$ and $a$. The next lemma guarantees that this equation has a solution in $\mathbb{N}$ and its proof provides us with a method for determining a set of generators for $I_S$.

\begin{lemma}
	\label{Lema-Sol-Natural}
	Consider the Diophantine equation in the variables $d$ and $a$ given by \begin{equation}
		\label{Dioph-Eq}
		n\cdot d + l \cdot a   =  m \cdot b_0,  
	\end{equation} for fixed $n,m \in \mathbb{N}$ and $b_0 \in \mathbb{N}\cup \{0\}$, $gcd(n,m)=1$, $1< n< m$, where $l= \pm 1$. This equation admits a solution in $\mathbb{N}$.
\end{lemma}
\begin{proof}
	Since $gcd(n,l)=1$, the Diophantine equation  $n\cdot d + l \cdot a = m \cdot b_0$ admits a solution in $\mathbb{Z}$. Moreover, writing $m=h_0n+r_1$ with $h_0,r_1\in \mathbb{N}$ and $0 < r_1 < n$, we have 
	$$n\cdot d + l \cdot a  =  (h_0b_0) \cdot n + r_1b_0.$$ There exist unique integers $k$ and $s$ such that $r_1b_0=k n +l \cdot s$, with $k \geq 0$ and $0 \leq s < n$, where for $l=-1$ we use the version of the Euclidean division theorem\footnote{A version of the Euclidean division theorem: Given the integers $c$ and $d>0$, there exist unique integers $k$ and $s$ such that $c=kd-s$ with $0 \leq s < d$.} given in the footnote. Thus, the previous equation admits a solution in $\mathbb{N}$ given by $d=h_0b_0+k$ and $a=s$, for $l=\pm 1$.
\end{proof}

\begin{definition}
	\label{Minimal-Solution}
	A solution $d_0,a_0 \in \mathbb{N}\cup \{0\}$ of (\ref{Dioph-Eq}) is called minimal, if $0 \leq a_0 < n$ and if $d_1,a_1\in \mathbb{N}\cup \{0\}$ is also a solution of the same equation, then $a_0 \leq a_1$.
\end{definition}

As consequence of the proof of the previous lemma, the minimal solution $d_0,a_0$ of (\ref{Dioph-Eq}) is unique.

By Proposition \ref{Eq-Toric-C4}, fixed $b=b_0$, if $d_1,a_1\in \mathbb{N}\cup \{0\}$ is a particular solution of the Diophantine equation (\ref{Dioph-Eq}), then by (\ref{nucleo}) we have the following correspondences: 
\begin{equation}
	\label{Corresp-Diop-Binomial}
	\begin{array}{ccccc}
		n\cdot d_1 - a_1  =  m \cdot b_0   & \leftrightarrow & Y^{a_1}W^{b_0}-X^{\lambda a_1}Z^{d_1} \in I_S & \leftrightarrow & \alpha_{a_1}\in Ker(\pi_S)\\
		n\cdot d_1 + a_1  =  m \cdot b_0   & \leftrightarrow & Y^{a_1}Z^{d_1}-X^{\lambda a_1}W^{b_0}\in I_S& \leftrightarrow & \beta_{a_1}\in Ker(\pi_S),\\
	\end{array}
\end{equation} where $\alpha_{a_1}=(-\lambda a_1,a_1,-d_1,b_0)$, and $\beta_{a_1}=(-\lambda a_1, a_1, d_1,-b_0)$.

\begin{remark}
	\label{Remark-Relev-binomials}
	The minimal solution (or any other solution) of the Diophantine equation (\ref{Dioph-Eq}) does not necessarily provide a relevant binomial in $I_S$, in the sense that it might not belong to the minimal set of generators of the ideal, since its associated pair of vectors in (\ref{Corresp-Diop-Binomial}) could be written as a linear combination of other vectors.   
\end{remark}

\vspace{0.2cm}
For $n,m\in \mathbb{N}$ co-primes, consider the successive division method 
\begin{equation}
	\label{Succ-Div-Meth}
	\begin{array}{rcl}
		m & = & h_0n+r_1 \\
		n & = & h_1r_1+r_2  \\
		r_1 & = & h_2r_2+r_3 \\
		& \vdots &\\
		r_{q-2}& = & h_{q-1}r_{q-1}+r_q \\
		r_{q-1}& = & h_qr_q,
	\end{array}
\end{equation}
where $h_0,h_{i},r_i \in \mathbb{N}$, for $i=1,\ldots, q$ with $r_0:=n$, $\ 0 < r_i < r_{i-1}$, and $r_q=gcd(n,m)=1$.

\begin{theorem}
	\label{Teo-Eq-Surf-C4}
	Let $X(S) =\mathcal{V}(I_S)$ be a toric surface in $\mathbb{C}^{4}$ generated by the semigroup $S = \left\langle (1,0), (\lambda,1), (0,n),(0,m) \right\rangle$ with $\lambda,n,m\in \mathbb{N}^*$, $1< n< m$, and $gcd(n,m)=1$. According to the notations in (\ref{Succ-Div-Meth}), the set ${\mathcal G}$ of the following binomials provides a generating set for $I_S$: $Y^n-X^{\lambda n}Z$, $Y^{r_1}Z^{h_0}-X^{\lambda r_1}W$, and  
	\begin{eqnarray}
		& Y^{a_{k,j}}W^{b_{k,j}}-X^{\lambda a_{k,j}}Z^{d_{k,j}}, & \mbox{for} \ \  k \ \ \mbox{odd},  \vspace{0.2cm} \label{Teo-eq-Type1}\\ & Y^{a_{k,j}}Z^{d_{k,j}}-X^{\lambda a_{k,j}}W^{b_{k,j}}, \label{Teo-eq-Type2} & \mbox{for} \ \  k \ \ \mbox{even},  \end{eqnarray} where for each fixed $k\in \{1, \ldots, q\}$ the exponents are given by $$\left\{\begin{array}{l}
		a_{k,j}=r_{k-1} - r_{k} \cdot j, \vspace{0.1cm}\\
		b_{k,j}=b_{k-1,h_{k-1}}\cdot j+b_{k-2,h_{k-2}} , \vspace{0.1cm}\\
		d_{k,j}= d_{k-1,h_{k-1}}\cdot j+d_{k-2,h_{k-2}},\\
	\end{array}\right.$$ for $j\in \{1, \ldots, h_k\}$, with $b_{-1,h_{-1}}:=0,b_{0,h_0}:=1$, $d_{-1,h_{-1}}:=1$ and $d_{0,h_0}=h_0$. 
\end{theorem}

Before proving this result, let us remark that by Proposition \ref{Eq-Toric-C4}, we need to determine a solution for the Diophantine equation in the variables $d$ and $a$ given by $n \cdot d + l \cdot a = m \cdot b$ for any fixed $b \in \mathbb{N}\cup \{0\}$ and $l=\pm 1$. For the proof of Theorem \ref{Teo-Eq-Surf-C4}, we determine the minimal solution of the previous equation for $b=b_{k,j}$ defined by a recursive relation, according to the parity of $k$. For the other values of $b \in \mathbb{N}$, we prove in Proposition  \ref{Prop-No-Generators} and Proposition \ref{Prop-No-Generators-2} that the binomials, obtained by the solution of its Diophantine equation, do not belong to the minimal generating set ${\mathcal G}$ of $I_S$. The complete proof of Theorem \ref{Teo-Eq-Surf-C4} depends on Propositions \ref{Rel-r1b_kh_k}, \ref{Binom-General-Solution}, \ref{Prop-No-Generators} and \ref{Prop-No-Generators-2} in Subsection \ref{Subsection-Complem}.

\begin{proof} 
	If $b=0$, then the equation $n \cdot d +l \cdot a=0$ has no solution in $\mathbb{N}$ for $l=1$ and for $l=-1$, we have $a=n\cdot d$. Thus $d=1$ and $a=n$ is the minimal solution for this equation. Consequently, the binomial $Y^n-X^{\lambda n}Z$ belongs to $I_S$ (see also the proof of Proposition \ref{Eq-Toric-C4}).
	
	For $b \geq 1$, take $b=b_{k,j}=b_{k-1,h_{k-1}}j+b_{k-2,h_{k-2}}$, with fixed $k \in \{1, \ldots, q\}$. We want to determine the values of the parameter $j\geq 1$ and the exponents $a_{k,j},d_{k,j} \in \mathbb{N}$ such that $Y^{a_{k,j}}W^{b_{k,j}}-X^{\lambda a_{k,j}}Z^{d_{k,j}}$ or $Y^{a_{k,j}}Z^{d_{k,j}}-X^{\lambda a_{k,j}}W^{b_{k,j}}$ are generators of $I_S$. By (\ref{Corresp-Diop-Binomial}), we need to obtain a solution of the Diophantine equation \begin{equation}
		\label{Diop-Eq-Type-1-2}
		n \cdot d_{k,j} +l \cdot a_{k,j} = m \cdot b_{k,j},
	\end{equation} for $l=\pm 1$. More precisely, we will obtain the minimal solution of this equation.
	
	The proof is by induction in $k$.
	
	\vspace{0.2cm}
	\noindent $\bullet \ $ For $k=1$ consider $b_{1,j}=b_{0,h_{0}}j+b_{-1,h_{-1}}=j,$ with $j \in \mathbb{N}$.
	
	\vspace{0.2cm}
	If $l=-1$, by (\ref{Succ-Div-Meth}) we can write (\ref{Diop-Eq-Type-1-2}) in the form
	$n\cdot d_{1,j}-a_{1,j} =(h_0n+r_1)\cdot j.$ Thus, $\ a_{1,j}=n(d_{1,j}-h_0j)-r_1 j$. Since $a_{1,j}\geq 0$ we have $d_{1,j}\neq h_0j$ and $n(d_{1,j}-h_0j) \geq r_1j >0$. Then, $d_{1,j}\geq h_0j+1$ and a solution for the previous equation is $$d_{1,j}=h_0j+1\ \ \ \mbox{and} \ \ \ a_{1,j}=n-r_1j,$$ where $a_{1,j}$ is non-negative for $1 \leq j \leq \left[\frac{n}{r_1}\right]=h_1$, and $[x]$ denotes the integral part of $x \in \mathbb{R}$. Moreover, $d_{1,j},a_{1,j}\in \mathbb{N}$ is the minimal solution (Definition \ref{Minimal-Solution}) of the Diophantine equation, since $r_1j$ is uniquely written as $r_1j=n-(n-r_1j)$ with $0 < s:=n-r_1j < n$ for $j\in \{1, \ldots, h_1\}$ (see the proof of Lemma \ref{Lema-Sol-Natural}). By hypothesis, $d_{0,h_0}=h_0$ and $d_{-1,h_{-1}}=1$, thus $d_{1,j}=h_0j+1=d_{0,h_0}j+d_{-1,h_{-1}}$. In this way, we get $h_1$ binomials in $I_S$ of type (\ref{Teo-eq-Type1}) and their corresponding vectors, as in (\ref{Corresp-Diop-Binomial}), given by
	\begin{equation}
		\label{Eq-Type1-k=1}
		Y^{n-r_1j}W^j-X^{\lambda(n-r_1j)}Z^{h_0j+1} \in I_S \ \leftrightarrow \ \alpha_{n-r_1j}=(-\lambda (n-r_1j),n-r_1j,-(h_0j+1),j)
	\end{equation}
	for $j\in \{1, \ldots, h_1\}$.
	
	For $l=1$ we have the equation $n\cdot d_{1,j} +a_{1,j}=(h_0n+r_1)\cdot j$. If $j=1$, then $d_{1,1}=h_0$ and $a_{1,1}=r_1$ is the minimal solution of the equation. Thus the binomial $Y^{r_1}Z^{h_0}-X^{\lambda r_1}W \in I_S$, and its corresponding vector is $\beta_{r_1}=(-\lambda r_1,r_1,h_0,-1)$. With the same previous argument for $l=-1$, we can write $a_{1,j}=n(h_0j-d_{1,j})+r_1 j$, for $1 < j \leq h_1$. In this case, the minimal solution is $d_{1,j}=h_0j$ and $a_{1,j}=r_1j$. However, the binomial $Y^{r_1j}Z^{h_0j}-X^{\lambda r_1j}W^j \in I_S \setminus {\mathcal G}$, since its corresponding vector satisfy $$(-\lambda r_1j,r_1j,h_0j,-j)=j \cdot (-\lambda r_1,r_1,h_0,-1)=j \cdot \beta_{r_1}$$
	see  Remark \ref{Remark-Relev-binomials}.
	
	\vspace{0.2cm}
	Thus, in addition to $Y^{n}-X^{\lambda n}Z$ and $Y^{r_1}Z^{h_0}-X^{\lambda r_1}W$, we obtained $h_1$ binomials of type (\ref{Teo-eq-Type1}), in which the last value of the parameter $b_{1,j}$ was for $j=h_1$. We want to determine for which values of $b_{2,j}$ greater than $b_{1,h_1}=h_1$, we have a solution for the Diophantine equation. 
	
	\vspace{0.2cm}
	\noindent $\bullet \ $ For $k=2$ consider $b_{2,j}=b_{1,h_1}j+b_{0,h_0}=h_1j+1$, with $j \geq 1$.
	
	\vspace{0.2cm}
	In this case, we start by analysing the equation (\ref{Diop-Eq-Type-1-2}) for $l=1$. In fact,
	$$\begin{array}{rcl}
		n\cdot d_{2,j}+a_{2,j}& = &(h_0n+r_1)\cdot (h_1j+1) \vspace{0.2cm}\\
		& = & h_0(h_1j+1)n +r_1h_1j +r_1 \vspace{0.2cm}\\
		& = & h_0(h_1j+1)n +(n-r_2)j +r_1 \vspace{0.2cm}\\
		& = & ((h_0h_1+1)j+h_0)n+r_1-r_2j,
	\end{array}$$ where the third equality follows from the second step of (\ref{Succ-Div-Meth}).
	
	Since $r_1-r_2j < n-r_2j < n$, we get $d_{2,j}=(h_0h_1+1)j+h_0$ and $a_{2,j}=r_1-r_2j$ as the minimal solution for the previous equation, where $a_{2,j}\geq 0$ if $1 \leq j\leq  \left[\frac{r_1}{r_2}\right]=h_2$. Moreover, $d_{2,j}=d_{1,h_1}j+d_{0,h_0}$. Therefore, we obtain $h_2$ binomials in $I_S$ of the type (\ref{Teo-eq-Type2}) given by
	$$Y^{r_1-r_2j}Z^{(h_0h_1+1)j+h_0}-X^{\lambda(r_1-r_2j)}W^{h_1j+1}, \ \ \mbox{for} \ j=1, \ldots, h_2.$$ 
	
	\vspace{0.2cm} For $l=-1$ in (\ref{Diop-Eq-Type-1-2}), as before we use the Euclidean division theorem to write $r_1-r_2j=n-(n-(r_1-r_2j))$, for $j \in \{1,\ldots, h_2\}$. Thus the minimal solution of this equation is given by $d_{2,j}=(h_0h_1+1)j+h_0+1$ and $a_{2,j}=n-(r_1-r_2j)$, for $j \in \{1,\ldots, h_2\}$. However, in this case, $Y^{a_{2,j}}W^{b_{2,j}}-X^{\lambda a_{2,j}}Z^{d_{2,j}}\in I_S \setminus {\mathcal G}$, since its corresponding vector as in (\ref{Corresp-Diop-Binomial}) satisfy
	$$\begin{array}{rcl}
		& &(-\lambda(n-(r_1-r_2j)),n-(r_1-r_2j),-[(h_0h_1+1)j+h_0+1],h_1j+1)\vspace{0.2cm}\\
		&=& (-\lambda(n-r_1),n-r_1,-(h_0+1),1)+j \cdot (-\lambda r_2,r_2,-(h_0h_1+1),h_1)\vspace{0.2cm}\\
		&=& \alpha_{n-r_1}+j \cdot  \alpha_{r_2} \ = \ \alpha_{n-r_1}+j \cdot \alpha_{n-r_1h_1}
	\end{array}$$ that is, it is a linear combination of two vectors in (\ref{Eq-Type1-k=1}) for $j=1$ and $j=h_1$, where $n-r_1h_1=r_2$ by (\ref{Succ-Div-Meth}). 
	Thus, we do not have in ${\mathcal G}$ binomials of type (\ref{Teo-eq-Type1}) for $b_{2,j}=h_1j+1$.
	
	\vspace{0.1cm} Suppose that the result is valid for any natural number $p$ with $1 \leq p \leq k$. 
	
	\vspace{0.2cm}
	\noindent $\bullet \ $ For fixed $k \in \mathbb{N}$ even and $j\geq 1$, consider $b_{k,j}=b_{k-1,h_{k-1}}\cdot j + b_{k-2,h_{k-2}}$.
	
	\vspace{0.2cm}
	Let us determine the minimal solution of the Diophantine equation
	\[\begin{array}{rcl}
		n \cdot d_{k,j}+a_{k,j}   &=&  (h_0n+r_1) \cdot b_{k,j} \\
		&= & h_0 b_{k,j}n+r_1b_{k-1,h_{k-1}}j+r_1b_{k-2,h_{k-2}}\ .
	\end{array}\] By the proof of Lemma \ref{Lema-Sol-Natural}, we need to express the term $r_1b_{k,h_k}$ as in the Euclidean division theorem. In Proposition \ref{Rel-r1b_kh_k}, we will prove for $k\geq 1$ that  \begin{equation}
		\label{Exp-r1b_k,h_k}
		r_1b_{k,h_k}=\left\{\begin{array}{l}
			n \cdot c_{k,h_k}+r_{k+1}, \ \mbox{if}\ k\ \mbox{is even}\\
			n \cdot c_{k,h_k}-r_{k+1}, \ \mbox{if}\ k\ \mbox{is odd},\\
		\end{array}\right.
	\end{equation} where $c_{k,h_k}=c_{k-1,h_{k-1}}\cdot h_k+c_{k-2,h_{k-2}}$, with $c_{-1,h_{-1}}:=1$ and $c_{0,h_0}:=0$. The elements $c_{k,h_k}$ and $r_{k+1}$ are unique, with $0 \leq r_{k+1}< n$ given in (\ref{Succ-Div-Meth}).
	
	By (\ref{Exp-r1b_k,h_k}) and the parity of $k-1$ and $k-2$, the previous Diophantine equation can be written of the form
	\begin{equation}
		\label{Exp-Step-k}\begin{array}{rcl}
			n \cdot d_{k,j}+a_{k,j}   &=& h_0b_{k,j}n+r_1b_{k-1,h_{k-1}}j+r_1b_{k-2,h_{k-2}}  \vspace{0.2cm}\\
			&= & h_0b_{k,j}n + (n \cdot c_{k-1,h_{k-1}}-r_{k})\cdot j+n c_{k-2,h_{k-2}}+r_{k-1} \vspace{0.2cm}\\
			&=& n \cdot (h_0b_{k,j}+c_{k-1,h_{k-1}} j+c_{k-2,h_{k-2}})+r_{k-1}-r_kj,
	\end{array}\end{equation}
	where $r_{k-1}-r_kj < n-r_kj< n$. Thus the minimal solution of this equation is \begin{equation}
		\label{Solution-k-even}d_{k,j}=h_0b_{k,j}+c_{k-1,h_{k-1}}j+c_{k-2,h_{k-2}} \ \ \mbox{and} \ \ a_{k,j}=r_{k-1}-r_kj\end{equation} in which $a_{k,j}\geq 0$ for $1\leq j \leq \left[\frac{r_{k-1}}{r_k}\right]=h_k$. In particular, \begin{equation}
		\label{Ind-Step-Equation}
		d_{k,h_k}=h_0b_{k,h_k}+c_{k-1,h_{k-1}}h_k+c_{k-2,h_{k-2}}=h_0b_{k,h_k}+c_{k,h_k}.   
	\end{equation}
	
	In Proposition \ref{Prop-No-Generators}, we will prove that for $k$ even, the binomial obtained by the solution of the equation $n \cdot d_{k,j}-a_{k,j} =b_{k,j}$ belongs to $I_S\setminus {\mathcal G}$, that is, we do not have in ${\mathcal G}$ binomials of type (\ref{Teo-eq-Type1}) for this parity of $k$.
	
	In the next step, we use the induction hypothesis to determine the solution of the Diophantine equation for $b=b_{k+1,j}$.
	
	\vspace{0.2cm}
	\noindent $\bullet \ $ For $b_{k+1,j}=b_{k,h_{k}}\cdot j + b_{k-1,h_{k-1}}$ for fixed $k \in \mathbb{N}$, where $k+1$ is odd and $j\geq 1$.
	
	\vspace{0.2cm} Our purpose is to obtain the minimal solution for the Diophantine equation $n \cdot d_{k+1,j}-a_{k+1,j}=m \cdot b_{k+1,j}$. Similarly to (\ref{Exp-Step-k}), we have 
	$$\begin{array}{rcl}
		n \cdot d_{k+1,j}-a_{k+1,j}   &=& h_0b_{k+1,j}n+r_1b_{k,h_{k}}j+r_1b_{k-1,h_{k-1}} \vspace{0.2cm}\\
		&=& n \cdot (h_0b_{k+1,j}+c_{k,h_{k}} j+c_{k-1,h_{k-1}})-(r_{k}-r_{k+1} \cdot j).
	\end{array}$$
	The condition $r_{k}-r_{k+1}\cdot j< n-r_{k+1}\cdot j < n$ implies that the minimal solution of the previous equation is
	\begin{equation}\label{Ex-d_k+1,j}
		d_{k+1,j}=h_0b_{k+1,j}+c_{k,h_{k}}j+c_{k-1,h_{k-1}} \ \ \mbox{and} \ \ a_{k+1,j}=r_{k}-r_{k+1}\cdot j,\end{equation} for $1\leq j \leq \left[\frac{r_{k}}{r_{k+1}}\right]=h_{k+1}$. By induction hypothesis, more precisely (\ref{Ind-Step-Equation}), we have
	\begin{equation}
		\label{Exp-d_k,j}
		\begin{array}{rcl}
			d_{k+1,j}  &=&  h_0b_{k+1,j}+c_{k,h_{k}} j+c_{k-1,h_{k-1}}\vspace{0.2cm}\\
			&=& h_0(b_{k,h_{k}}\cdot j + b_{k-1,h_{k-1}})+c_{k,h_{k}} j+c_{k-1,h_{k-1}}\vspace{0.2cm}\\
			&=& (h_0b_{k,h_{k}}+c_{k,h_{k}})\cdot j + h_0 b_{k-1,h_{k-1}}+c_{k-1,h_{k-1}}\vspace{0.2cm}\\
			&=& d_{k,h_k}\cdot j + d_{k-1,h_{k-1}}
	\end{array}\end{equation} which guarantees the recurrence relation for $d_{k+1,j}$.
	
	Similarly, by Proposition \ref{Prop-No-Generators}, we do not have in ${\mathcal G}$ binomials of type (\ref{Teo-eq-Type2}) for $k+1$ odd. Moreover, we will prove in Proposition \ref{Prop-No-Generators-2}  that the binomials (\ref{Eq-Type1}) and (\ref{Eq-Type2}) are not in ${\mathcal G}$, for any $b \in \mathbb{N}$ with $b \neq b_{k,j}$ for $k \in \{1, \ldots, q\}$ and $j \in \{1, \ldots, h_k\}$.
	
	Finally, note that this recursive process to obtain a set of generators of $I_S$ stops, since the successive division method is finite. 
\end{proof}

\begin{corollary} \label{Cor-No-Equations}
	Let $X(S)=\mathcal{V}(I_S) \subset \mathbb{C}^4$ be a toric surface generated by the semigroup $S = \langle (1,0), (\lambda,1), (0,n),(0,m) \rangle$ with $n,m\in \mathbb{N}$, $1< n< m$, and $gcd(n,m)=1$. The ideal $I_S$ has $\sum_{k=1}^{q}h_k+2$ generators, given by $Y^n-X^{\lambda n}Z$ and $Y^{r_1}Z^{h_0}-X^{\lambda r_1}W$  plus the binomials of the successive division method (\ref{Succ-Div-Meth}) as follows $$h_{k} \ \ \mbox{binomials of} \ \ \left\{\begin{array}{ll}
		\mbox{type} \  (\ref{Teo-eq-Type1}) & \mbox{for} \ \  k \ \ \mbox{odd} \\
		\mbox{type} \  (\ref{Teo-eq-Type2}) & \mbox{for} \ \  k \ \ \mbox{even},
	\end{array}\right.$$
	with $k=1,\ldots,  q$. 
\end{corollary}

\begin{example}
	We apply the previous theorem for some classes of toric surfaces generated by the following semigroups:
	
	\vspace{0.2cm}
	\noindent $(1)\ $ $ \ S=\langle (1,0),(3,1),(0,225),(0,328)\rangle$.
	
	\vspace{0.2cm}
	
	By Corollary \ref{Cor-No-Equations}, the ideal $I_S$ has $\sum_{k=1}^6h_k+2=16$ generators. From Theorem \ref{Teo-Eq-Surf-C4} the elements in $I_S$ are given by $Y^n-X^{\lambda n}Z=Y^{225}-X^{675}Z$, $Y^{r_1}Z^{h_0}-X^{\lambda r_1}W=Y^{103}Z-X^{309}W,$ $Y^{a_{k,j}}W^{b_{k,j}}-X^{\lambda a_{k,j}}Z^{d_{k,j}}$, and $Y^{a_{k,j}}Z^{d_{k,j}}-X^{\lambda a_{k,j}}W^{b_{k,j}}$ in which the exponents $a_{k,j}, b_{k,j}$ and $d_{k,j}$ are described in the following table, with $b_{0,h_0}=1=d_{-1,h_{-1}}$, $b_{-1,h_{-1}}=0$, and $d_{0,h_0}=h_0=1$.
	
	\begin{figure}[h]
		\hspace{1cm}
		\begin{minipage}[c]{0.6cm}
			\scriptsize{$$\begin{array}{rl}
					328=& \hspace{-0.3cm} 1\cdot 225 + 103\\
					225 = & \hspace{-0.3cm}  2 \cdot 103 + 19\\
					103 = & \hspace{-0.3cm}  5 \cdot 19 + 8\\
					19 = & \hspace{-0.3cm} 2 \cdot 8 + 3\\
					8 = & \hspace{-0.3cm} 2 \cdot 3 + 2\\
					3 = & \hspace{-0.3cm} 1 \cdot 2+ 1\\
					2 =& \hspace{-0.3cm} 2 \cdot 1
				\end{array}$$
				
				$$\begin{array}{l}
					h_0=1,\ r_1=103,\\
					h_1=2,\ r_2=19,\\
					h_2=5,\ r_3=8,\\
					h_3=2,\ r_4=3, \\
					h_4=2,\ r_5=2, \\ 
					h_5=1,\ r_6=1,\\ 
					h_6=2 
				\end{array}$$}
		\end{minipage}
		\hspace{2.1cm}
		\begin{minipage}[c]{4cm}
			\scriptsize{\[
				\begin{array}{l|l}
					\hspace{1cm}a_{k,j}, \ b_{k,j},\ d_{k,j} & \mbox{Binomials in}\ I_S  \\ \hline
					a_{1,j}=n-r_1j=225-103j, \ \ j=1,2=h_1 & \\
					b_{1,j}=b_{0,h_0}j+b_{-1,h_{-1}}=j & Y^{122}W-X^{366}Z^2 \\
					d_{1,j}=d_{0,h_0}j+d_{-1,h_{-1}}=h_0j+1=j+1 & Y^{19}W^2-X^{57}Z^3 \\ \hline
					&  Y^{84}Z^4-X^{252}W^3\\
					a_{2,j}=r_1-r_2j=103-19j,  \ \ j=1,\ldots,5=h_2 & Y^{65}Z^7-X^{195}W^5 \\
					b_{2,j}=b_{1,h_1}j+b_{0,h_0}=h_1j+1=2j+1 &  Y^{46}Z^{10}-X^{138}W^7 \\
					d_{2,j}=d_{1,h_1}j+d_{0,h_0}=(h_0h_1+1)j+h_0=3j+1 & Y^{27}Z^{13}-X^{81}W^9 \\
					& Y^{8}Z^{16}-X^{24}W^{11}\\ \hline
					a_{3,j}=r_2-r_3j=19-8j, \ \ j=1,2=h_3 &  \\
					b_{3,j}=b_{2,h_2}j+b_{1,h_1}=(2h_2+1)j+h_1=11j+2 & Y^{11}W^{13}-X^{33}Z^{19}\\
					d_{3,j}=d_{2,h_2}j+d_{1,h_1}=(3h_2+1)j+h_1+1=16j+3 & Y^{3}W^{24}-X^{9}Z^{35} \\ \hline
					a_{4,j}=r_3-r_4j=8-3j, \ \  j=1,2=h_4 & \\
					b_{4,j}=b_{3,h_3}j+b_{2,h_2}=(11h_3+2)j+2h_2+1=24j+11 & Y^5Z^{51}-X^{15}W^{35}\\
					d_{4,j}=d_{3,h_3}j+d_{2,h_2}=(16h_3+3)j+3h_2+1=35j+16 & Y^2Z^{86}-X^{6}W^{59} \\ \hline
					a_{5,j}=r_4-r_5j=3-2j, \ \ j=1=h_5 &\\
					b_{5,j}=b_{4,h_4}j+b_{3,h_3}=(24h_4+11)j+11h_3+2=59j+24& YW^{83}-X^{3}Z^{121} \\
					d_{5,j}=d_{4,h_4}j+d_{3,h_3}=(35h_4+16)j+16h_3+3=86j+35 & \\ \hline
					a_{6,j}=r_5-r_6j=2-j, \ \ j=1,2=h_6 &\\
					b_{6,j}=b_{5,h_5}j+b_{4,h_4}=(59h_5+24)j+24h_4+11=83j+59& YZ^{207}-X^{3}W^{142} \\
					d_{6,j}=d_{5,h_5}j+d_{4,h_4}=(86h_5+35)j+35h_4+16=121j+86 & Z^{328}-W^{225} \\ \hline
				\end{array}\]}
		\end{minipage}
	\end{figure}

	\vspace{0.2cm}
	\noindent $(2)\ $ $ \ S=\langle(1,0),(\lambda,1),(0,n),(0,m)\rangle$, with $\lambda,n,m \in \mathbb{N}$ and $1< n< m$.
	
	\vspace{0.2cm}
	\begin{enumerate}
		\item[(a)] If $m=h_0 n +1$ for some $h_0 \in \mathbb{N}$.
		
		\vspace{0.2cm} By the successive division method, we have 
		$$\begin{array}{l}
			m=h_0\cdot n + 1\\
			\ n= n \cdot 1,
		\end{array}$$
		that is, $r_1=1$ and $h_1=n$. Thus, the Corollary \ref{Cor-No-Equations} guarantees that  $I_S$ has $h_1+2=n+2$ generators, described by Theorem \ref{Teo-Eq-Surf-C4} of the form
		$$I_S=\langle Y^n-X^{\lambda n} Z, \ YZ^{h_0}-X^{\lambda}W, \ Y^{n-j}W^j-X^{\lambda(n-j)}Z^{h_0j+1}, \ \mbox{for}\ j=1, \ldots, h_1\rangle.$$
		In this case, the only binomial of type (\ref{Teo-eq-Type2}) is  $YZ^{h_0}-X^{\lambda}W$.
		
		\vspace{0.2cm}
		\item[(b)] If $m=h_0n+n-1$ for some $h_0 \in \mathbb{N}$.
		
		\vspace{0.2cm} For this class of semigroups, we have by (\ref{Succ-Div-Meth}) 
		$$\begin{array}{rcl}
			m&=&h_0\cdot n + n-1\\
			n&=& 1 \cdot (n-1)+1 \\
			n-1&=&(n-1)\cdot 1,
		\end{array}$$
		that is, $r_1=n-1$, $h_1=1$, $r_2=1$, and $h_2=n-1$. Thus, the ideal $I_S$ has $h_1+h_2+2=n+2$ generators. By Corollary \ref{Cor-No-Equations}, except the binomials $Y^n-X^{\lambda n} Z$ and $Y^{n-1}Z^{h_0}-X^{\lambda(n-1)}W$ we have $h_1=1$ binomial of type (\ref{Teo-eq-Type1}) and $h_2=n-1$ of type (\ref{Teo-eq-Type2}). More precisely, by Theorem \ref{Teo-Eq-Surf-C4}  we obtain $$\begin{array}{rcl}
			I_S & = & \langle Y^{n}-X^{\lambda n}Z, \ Y^{n-1}Z^{h_0}-X^{\lambda(n-1)}W, \ YW-X^{\lambda}Z^{h_0+1}, \vspace{0.2cm}\\
			&& Y^{n-1-j}Z^{(h_0+1)j+h_0}-X^{\lambda(n-1-j)}W^{j+1}; \ \mbox{for}\ j=1, \ldots, h_2 \rangle .
		\end{array}$$
	\end{enumerate}
\end{example}

\subsection{Complementary results} \label{Subsection-Complem}


\begin{proposition}
	\label{Rel-r1b_kh_k}
	Given $k \in \{1,\ldots, q\}$, consider $b_{k,h_{k}}=b_{k-1,h_{k-1}}\cdot h_k+b_{k-2,h_{k-2}}$, with $b_{0,h_0}=1$ and $b_{-1,h_{-1}}=0$. Using the notations of (\ref{Succ-Div-Meth}), we have  $$r_1b_{k,h_{k}}=\left\{\begin{array}{l}
		n \cdot c_{k,h_k}+r_{k+1}, \ \mbox{if}\ k\ \mbox{is even}\\
		n \cdot c_{k,h_k}-r_{k+1}, \ \mbox{if}\ k\ \mbox{is odd},\\
	\end{array}\right.$$ where $0 < r_q < r_{q-1}< \ldots < r_2 < r_1 < n$, and $c_{k,h_k}$ satisfy the recursive relation $c_{k,h_k}=c_{k-1,h_{k-1}}\cdot h_k+c_{k-2,h_{k-2}}$, with $c_{0,h_0}:=0$ and $c_{-1,h_{-1}}:=1$. Moreover, $c_{k,h_k}$ and $r_{k+1}$ are unique.
\end{proposition}
\begin{proof}
	The proof is by induction in $k$. By (\ref{Succ-Div-Meth}), we obtain for $k=1$ and $k=2$, respectively, that $r_1b_{1,h_1}=r_1h_1=n-r_2=nc_{1,h_1}-r_2$ and $r_1b_{2,h_2}=r_1(b_{1,h_1}\cdot h_2+b_{0,h_0})=r_1(h_1h_2+1)=(n-r_2)h_2+r_1=nh_2-(r_1-r_3)+r_1=nh_2+r_3=nc_{2,h_2}+r_3$, since $c_{2,h_2}=c_{1,h_1}\cdot h_2+c_{0,h_0}=h_2$. 
	
	Suppose that the result is valid for any natural number $p$ with $1 \leq p \leq k$. If $k+1$ is odd, by induction hypothesis, we have 
	$$\begin{array}{rcl}
		r_1b_{k+1,h_{k+1}} &= & r_1(b_{k,h_k}\cdot h_{k+1}+b_{k-1,h_{k-1}}) \vspace{0.1cm} \\
		&= &(n \cdot c_{k,h_k}+r_{k+1})h_{k+1}+n \cdot c_{k-1,h_{k-1}}-r_{k} \vspace{0.1cm}\\
		&= &n \cdot (c_{k,h_k}\cdot h_{k+1}+c_{k-1,h_{k-1}})+h_{k+1}r_{k+1}-r_{k} \vspace{0.1cm}\\
		&= &n \cdot c_{k+1,h_{k+1}}-r_{k+2},
	\end{array}$$ since by (\ref{Succ-Div-Meth}) we write $h_{k+1}r_{k+1}=r_{k}-r_{k+2}$. Similarly for $k+1$ even. The uniqueness of $c_{k,h_k}$ and $r_{k+1}$ follows from the Euclidean division theorem.
\end{proof}

\quad
In the next result, we prove that for each fixed $b$, if $d_1,a_1$ and $d_2,a_2$ are two solutions in $\mathbb{N}\cup \{0\}$ of the same Diophantine equation, then there is a relation between the corresponding vectors associated with these solutions (see (\ref{Corresp-Diop-Binomial})). 

Let $d_1,a_1 \in \mathbb{N}\cup\{0\}$ be any solution of (\ref{Dioph-Eq}). It is well known that the other solutions of the same equation can be given by 
$$d=d_1-l \cdot t  \hspace{0.3cm}\mbox{and}\hspace{0.3cm} a=a_1+n \cdot t $$ for all $t \in \mathbb{Z}$ with $l=\pm 1$.

\begin{proposition}
	\label{Binom-General-Solution}
	Let $d_1,a_1$ and $d_2,a_2$ in $\mathbb{N}\cup \{0\}$ be two solutions of the Diophantine equation $ n \cdot d + l \cdot a = m \cdot b_0$, for some fixed $b_0 \in \mathbb{N}$ and $l=\pm 1$. If $d_2=d_1-l \cdot t_0$ and $a_2=a_1+n \cdot t_0$, for some $t_0 \in \mathbb{N}$, then its corresponding vector, as in (\ref{Corresp-Diop-Binomial}), satisfy the following relation:
	\begin{enumerate}
		\item[(i)] $\alpha_{a_2}=\alpha_{a_1}+t_0 \cdot \alpha_n$, for $l=-1$; \vspace{0.2cm}
		\item[(ii)] $\beta_{a_2}=\beta_{a_1}+t_0 \cdot \alpha_n$, for $l=1$,
	\end{enumerate} where $\alpha_n = (-\lambda n,n,-1,0)$ is associated with the binomial $Y^n-X^{\lambda n}Z$ in Theorem \ref{Teo-Eq-Surf-C4}.
\end{proposition}
\begin{proof}
	\noindent (i) For $l=-1$, we have $d_2=d_1+t_0$ and $a_2=a_1+n \cdot t_0$, for some $t_0 \in \mathbb{N}$. By (\ref{Corresp-Diop-Binomial}), the corresponding vectors associated with $Y^{a_1}W^{b_0}-X^{\lambda a_1}Z^{d_1}$ and $Y^{a_2}W^{b_0}-X^{\lambda a_2}Z^{d_2}$ satisfy $$\begin{array}{rcl}
		\alpha_{a_2} & = & (-\lambda a_2,a_2,-d_2,b_0)= (-\lambda(a_1+n\cdot t_0),a_1+n\cdot t_0,-(d_1+t_0),b_0)\\
		&=& (-\lambda a_1,a_1,-d_1,b_0) +t_0 \cdot (-\lambda n,n,-1,0) \\
		& = & \alpha_{a_1}+t_0\cdot \alpha_n.
	\end{array}$$
	
	\noindent (ii) For $l=1$, we have $d_2=d_1-t_0$ and $a_2=a_1+n \cdot t_0$, for some $t_0 \in \mathbb{N}$. Thus, the vectors associated with $Y^{a_1}Z^{d_1}-X^{\lambda a_1}W^{b_0}$ and $Y^{a_2}Z^{d_2}-X^{\lambda a_2}W^{b_0}$ satisfy
	$$\begin{array}{rcl}
		\beta_{a_2} & = & (-\lambda a_2,a_2,d_2,-b_0) =  (-\lambda(a_1+n\cdot t_0),a_1+n\cdot t_0,d_1-t_0,-b_0)  \\
		&=& (-\lambda a_1,a_1,d_1,-b_0) +t_0 \cdot (-\lambda n,n,-1,0)  \\
		&=& \beta_{a_1}+t_0 \cdot \alpha_n.
	\end{array}$$
\end{proof}

Let us recall that ${\mathcal G}$ denotes the set $${\mathcal G} = \{Y^{a}W^{b}-X^{\lambda a}Z^{d}, Y^{a}Z^{d}-X^{\lambda a}W^{b} \ \mbox{as in Theorem \ref{Teo-Eq-Surf-C4}}\}.$$

In Theorem \ref{Teo-Eq-Surf-C4}, for $b_0=b_{k,j}$ in (\ref{Dioph-Eq}) we determine binomials of type (\ref{Teo-eq-Type1}) for $k$ odd, and binomials of type (\ref{Teo-eq-Type2}) for $k$ even. Let us prove that only these binomials belong to ${\mathcal G}$.  

For the next results, we need to take the vectors associated with the binomials $Y^n-X^{\lambda n}Z$, $Y^{a_{k,j}}W^{b_{k,j}}-X^{\lambda a_{k,j}}Z^{d_{k,j}}$, $Y^{r_1}Z^{h_0}-X^{\lambda r_1}W$, and $Y^{a_{k,j}}Z^{d_{k,j}}-X^{\lambda a_{k,j}}W^{b_{k,j}}$ in Theorem \ref{Teo-Eq-Surf-C4}, given respectively by
\begin{equation}
	\label{Vector-a_{k,j}}
	\begin{array}{c}
		\alpha_n=(-\lambda n, n,-1,0), \     \alpha_{a_{k,j}}=(-\lambda a_{k,j},a_{k,j},-d_{k,j},b_{k,j}),  \ \mbox{for}\ k \ \mbox{odd} \vspace{0.2cm} \\
		\beta_{r_1}=(-\lambda r_1,r_1,h_0,-1), \ \beta_{a_{k,j}}=(-\lambda a_{k,j},a_{k,j},d_{k,j},-b_{k,j}),  \ \mbox{for}\ k \ \mbox{even}, \\
	\end{array}
\end{equation} where $a_{k,j}=r_{k-1}-r_kj$, for $k \in \{1, \ldots, q\}$ and $j \in \{1, \ldots, h_k\}$. In particular, for $j=h_k$ it follows from (\ref{Succ-Div-Meth}) that $r_{k-1}=h_kr_k+r_{k+1}$. Thus
\begin{equation}
	\label{Vector-a_{k,h_k}}
	\begin{array}{l}
		\alpha_{a_{k,h_k}}=(-\lambda r_{k+1},r_{k+1},-d_{k,h_k},b_{k,h_k}):=\alpha_{r_{k+1}} \vspace{0.2cm} \\
		\beta_{a_{k,h_k}}=(-\lambda r_{k+1},r_{k+1},d_{k,h_k},-b_{k,h_k}):=\beta_{r_{k+1}}.
	\end{array}
\end{equation}

\begin{proposition}
	\label{Prop-No-Generators}
	Consider in (\ref{Dioph-Eq}) the integer $b_0=b_{k,j}=b_{k-1,h_{k-1}}\cdot j+b_{k-2,h_{k-2}}$, with $b_{0,h_0}=1$ and $b_{-1,h_{-1}}=0$, for some $k \in \{1, \ldots, q\}$ and $j \in \{1, \ldots, h_k\}$ as in Theorem \ref{Teo-Eq-Surf-C4}. Let $d,a \in \mathbb{N}\cup \{0\}$ be the minimal solution of the Diophantine equation $n \cdot d+l \cdot a = m\cdot b_{k,j}$, with $l=\pm 1$. 
	\begin{enumerate}
		\item[(i)] If $k$ is odd, then $Y^aZ^d-X^{\lambda a}W^{b_{k,j}} \in I_S \setminus {\mathcal G}$, except for $k=j=1$; \vspace{0.1cm}
		\item[(ii)] If $k$ is even, then $Y^aW^{b_{k,j}}-X^{\lambda a}Z^d \in I_S \setminus {\mathcal G}$.
	\end{enumerate}
\end{proposition}
\begin{proof}
	In the proof of Theorem \ref{Teo-Eq-Surf-C4}, we get for $k=l=1$ that the binomial $Y^aZ^d-X^{\lambda a}W^{b_{1,j}} \in I_S \setminus {\mathcal G}$ for each $j\in \{2, \ldots, h_1\}$, and for $k=2$ and $l=-1$ that $Y^aW^{b_{2,j}}-X^{\lambda a}Z^d \in I_S \setminus {\mathcal G}$ for each $j\in \{1, \ldots, h_2\}$.
	
	\vspace{0.1cm}
	\noindent (i) Our purpose is to determine, for $k$ an odd integer, the minimal solution to the equation $n \cdot d + a = m \cdot b_{k,j}$, for some $j\in \{1, \ldots, h_k\}$.  With a similar computation as in (\ref{Exp-Step-k}) and by Proposition \ref{Rel-r1b_kh_k}, we have 
	$$\begin{array}{rcl}
		n \cdot d + a & =  & h_0b_{k,j}n+r_1b_{k-1,h_{k-1}}\cdot j + r_1b_{k-2,h_{k-2}}\vspace{0.2cm} \\
		&=& h_0b_{k,j}n+(nc_{k-1,h_{k-1}}+r_k)j + (nc_{k-2,h_{k-2}}-r_{k-1})\vspace{0.2cm} \\
		&= & n \cdot (h_0b_{k,j}+c_{k-1,h_{k-1}}\cdot j + c_{k-2,h_{k-2}}) -n+(n - (r_{k-1}-r_kj))\vspace{0.2cm} \\
		&= & n \cdot (d_{k,j}-1) +(n- (r_{k-1}-r_kj)),
	\end{array}$$ where, in the third equality, we use the Euclidean division theorem which guarantees the uniqueness of the expression $-(r_{k-1}-r_kj)=(-1)n+(n-(r_{k-1}-r_kj))$, with $0 \leq n-(r_{k-1}-r_kj) < n$, since $0 \leq r_{k-1}-r_kj < n$ for $j \in \{1,\ldots, h_k\}$, and in the fourth equality we use the expression (\ref{Solution-k-even}). Thus, the minimal solution of the previous equation is
	$$d=d_{k,j}-1 \hspace{0.3cm}\mbox{and} \hspace{0.3cm}a=n-(r_{k-1}-r_kj),$$ where $d_{k,j}=d_{k-1,h_{k-1}}\cdot j + d_{k-2,h_{k-2}}$. Thus, $Y^aZ^d-X^{\lambda a}W^{b_{k,j}} \in I_S$ for $k$ odd.
	
	The vector $\beta_a$ associated with the previous binomial can be obtained as a linear combination of vectors associated with the binomials given in Theorem \ref{Teo-Eq-Surf-C4}. In fact, 
	\begin{equation}
		\label{PairVectors-Sol-1}
		\begin{array}{c}
			\beta_a=(-\lambda a,a,d,-b_{k,j})= (-\lambda(n-r_{k-1}),n-r_{k-1},d_{k-2,h_{k-2}}-1,-b_{k-2,h_{k-2}})\vspace{0.2cm}\\
			\hspace{2.3cm}+ j \cdot (-\lambda r_{k},r_{k},d_{k-1,h_{k-1}},-b_{k-1,h_{k-1}}).\end{array}\end{equation}
	
	Note that we can write \begin{equation}
		\label{Rel-n-r_k-1} 
		\begin{array}{rcl}
			n-r_{k-1} & = & (n-r_2)+(r_2-r_4)+\ldots + (r_{k-5}-r_{k-3})+(r_{k-3}-r_{k-1})  \\
			& = & \displaystyle\sum_{i=0}^{\frac{k-3}{2}}(r_{2i}-r_{2i+2})= \displaystyle\sum_{i=0}^{\frac{k-3}{2}}h_{2i+1}r_{2i+1}, 
		\end{array}
	\end{equation} where the last equality follows from (\ref{Succ-Div-Meth}).
	
	We are going to prove the following:
	
	\vspace{0.2cm}
	\noindent {\sc Claim 1:} If $k$ is an odd integer, with $k \geq 3$, then
	\begin{equation}
		\label{Rel-b_k-2,d_k-2}
		b_{k-2,h_{k-2}}=\displaystyle\sum_{i=0}^{\frac{k-3}{2}}h_{2i+1}\cdot b_{2i,h_{2i}} \hspace{0.2cm}\mbox{and} \hspace{0.2cm} d_{k-2,h_{k-2}}-1=\displaystyle\sum_{i=0}^{\frac{k-3}{2}}h_{2i+1}\cdot d_{2i,h_{2i}}.
	\end{equation}
	
	The proof is by induction on $k$.
	
	Note that  the first relation  in (\ref{Rel-b_k-2,d_k-2}) is verified for $k=3$, since $h_1b_{0,h_0}=h_1 \cdot 1=h_1=b_{1,h_1}$. Suppose that the equality is valid for $k$. In step $k+2$ (which is an odd integer), by induction hypothesis, we get
	$$h_1b_{0,h_0}+h_3b_{2,h_2}+\ldots + h_{k-2}b_{k-3,h_{k-3}}+h_kb_{k-1,h_{k-1}}=b_{k-2,h_{k-2}}+h_kb_{k-1,h_{k-1}}=b_{k,h_{k}}.$$
	
	For the other relation, we have for $k=3$ that $h_1d_{0,h_0}=h_1 \cdot h_0=h_0h_1+1-1=d_{1,h_1}-1$. Then, 
	$$\begin{array}{l}
		h_1d_{0,h_0}+h_3d_{2,h_2}+\ldots + h_{k-2}d_{k-3,h_{k-3}}+h_kd_{k-1,h_{k-1}} \vspace{0.2cm}\\
		=d_{k-2,h_{k-2}}-1+h_kd_{k-1,h_{k-1}}=d_{k-1,h_{k-1}}\cdot h_k+d_{k-2,h_{k-2}}-1=d_{k,h_{k}}-1.\end{array}$$
	\vspace{0.2cm}
	
	Consider all even indices less than $k$ given of the form $u=2i$ for $i \in \left\{0,\ldots, \frac{k-1}{2}\right\}$. By (\ref{Vector-a_{k,h_k}}), for each $u$ and $j=h_u$ we have 
	\begin{equation}
		\label{Binom-Even-Indices}
		\beta_{a_{u,h_u}}=(-\lambda r_{2i+1},r_{2i+1},d_{2i,h_{2i}},-b_{2i,h_{2i}}):= \beta_{r_{2i+1}}.
	\end{equation} 
	
	Therefore, by (\ref{Rel-n-r_k-1}), (\ref{Rel-b_k-2,d_k-2}), and (\ref{Binom-Even-Indices}), we can write (\ref{PairVectors-Sol-1}) of the form
	{\scriptsize $$
		\begin{array}{rcl}
			\beta_a &=& \left(-\lambda \left(\displaystyle\sum_{i=0}^{\frac{k-3}{2}}h_{2i+1} r_{2i+1}\right),\displaystyle\sum_{i=0}^{\frac{k-3}{2}}h_{2i+1} r_{2i+1},\displaystyle\sum_{i=0}^{\frac{k-3}{2}}h_{2i+1} d_{2i,h_{2i}},-\displaystyle\sum_{i=0}^{\frac{k-3}{2}}h_{2i+1} b_{2i,h_{2i}}\right)+ j \cdot \beta_{r_k}\vspace{0.2cm}\\
			&=& \displaystyle\sum_{i=0}^{\frac{k-3}{2}} h_{2i+1}\cdot (-\lambda r_{2i+1},r_{2i+1},d_{2i,h_{2i}},-b_{2i,h_{2i}}) + j \cdot \beta_{r_k}.\\
		\end{array}$$}
	
	Thus, the vector $\beta_a$ associated with the minimal solution of $n \cdot d + a = m \cdot b_{k,j}$, for $k$ an odd integer, satisfy
	\begin{equation}
		\label{Exp-beta_a-k-odd}
		\beta_a= \displaystyle\sum_{i=0}^{\frac{k-3}{2}} h_{2i+1}\cdot \beta_{r_{2i+1}}+ j \cdot \beta_{r_k},   
	\end{equation}
	that is, it is written as a combination of vectors of binomials given in Theorem \ref{Teo-Eq-Surf-C4}. Thus, $Y^aZ^d-X^{\lambda a}W^{b_{k,j}} \in I_S \setminus {\mathcal G}$.
	
	\vspace{0.2cm}
	\noindent (ii) For $k$ an even integer, the equation $n \cdot d - a = m \cdot b_{k,j}$ for some $j\in \{1, \ldots, h_k\}$, can be written, similarly to the previous item, as
	$$\begin{array}{rcl}
		n \cdot d - a & =  & h_0b_{k,j}n+r_1b_{k-1,h_{k-1}}\cdot j + r_1b_{k-2,h_{k-2}}\vspace{0.2cm} \\
		&=& h_0b_{k,j}n+(nc_{k-1,h_{k-1}}-r_k)j + (nc_{k-2,h_{k-2}}+r_{k-1})\vspace{0.2cm} \\
		&= & n \cdot (h_0b_{k,j}+c_{k-1,h_{k-1}}\cdot j + c_{k-2,h_{k-2}}) +n-(n - (r_{k-1}-r_kj))\vspace{0.2cm} \\
		&= & n \cdot (d_{k,j}+1) -(n- (r_{k-1}-r_kj)),
	\end{array}$$
	where in the last equality we use (\ref{Solution-k-even}). Thus, the minimal solution of the previous equation is 
	\begin{equation}
		\label{Irrel-Solution-k-even}
		d=d_{k,j}+1 \hspace{0.3cm}\mbox{and} \hspace{0.3cm}a=n-(r_{k-1}-r_kj),  
	\end{equation} with $d_{k,j}=d_{k-1,h_{k-1}}\cdot j + d_{k-2,h_{k-2}}$, for $j\in \{1, \ldots, h_k\}$. In this case, 
	\begin{equation}
		\label{Rel-n-r_k-2} \begin{array}{rcl}
			n-r_{k-1} & = & (n-r_1)+(r_1-r_3)+\ldots + (r_{k-3}-r_{k-1})  \\
			& = & (n-r_1)+\displaystyle\sum_{i=1}^{\frac{k-2}{2}}(r_{2i-1}-r_{2i+1})= n-r_1+\displaystyle\sum_{i=1}^{\frac{k-2}{2}}h_{2i}r_{2i}.
	\end{array}\end{equation}
	
	The proof of the Claim $2$ follows similarly to Claim $1$.
	
	\vspace{0.2cm}
	\noindent {\sc Claim 2:} If $k$ is an even integer, with $k \geq 2$, then
	$$
	b_{k-2,h_{k-2}}=\displaystyle\sum_{i=1}^{\frac{k-2}{2}}h_{2i}\cdot b_{2i-1,h_{2i-1}} +1 \hspace{0.2cm}\mbox{and} \hspace{0.2cm} d_{k-2,h_{k-2}}+1= h_0+1+\displaystyle\sum_{i=1}^{\frac{k-2}{2}}h_{2i}\cdot d_{2i-1,h_{2i-1}}.$$
	
	Consider all odd indices $u$ less than $k$, that is, $u=2i-1$ with $i \in \left\{1,\ldots, \frac{k}{2}\right\}$. In (\ref{Vector-a_{k,j}}) and (\ref{Vector-a_{k,h_k}}), for each $u$ and $j \in \{1, \ldots, h_u\}$ take the vectors \begin{equation}
		\label{Rel-PairVectors-keven}
		\begin{array}{l}
			\alpha_{a_{1,1}}=(-\lambda (n-r_1),n-r_1,-(h_0+1),1):=\alpha_{n-r_1}\ \ \mbox{and}\vspace{0.2cm}\\ \alpha_{a_{u,h_u}}=(-\lambda r_{u+1},r_{u+1},-d_{2i-1,h_{2i-1}},b_{2i-1,h_{2i-1}}) =\alpha_{r_{2i}},
		\end{array}
	\end{equation} 
	where in the second vector we consider $u \neq 1$, and $r_{u+1}=r_{2i}$.
	By the Claim 2, (\ref{Rel-n-r_k-2}),  (\ref{Rel-PairVectors-keven}), and $d$, $a$ given in (\ref{Irrel-Solution-k-even}), we can express the vector associated with the binomial $Y^aW^{b_{k,j}}-X^{\lambda a}Z^d$ in the form
	{\small $$
		\begin{array}{rcl}
			\alpha_a &= & (-\lambda a, a , -d,b_{k,j}) \vspace{0.1cm}\\
			& = & (-\lambda(n-r_{k-1}),n-r_{k-1},-(d_{k-2,h_{k-2}}+1),b_{k-2,h_{k-2}}) +j\cdot (-\lambda r_k,r_k,-d_{k-1,h_{k-1}},b_{k-1,h_{k-1}})\vspace{0.2cm}\\ 
			&=& (-\lambda (n-r_1),n-r_1,-(h_0+1),1)+\displaystyle\sum_{i=1}^{\frac{k-2}{2}}h_{2i} \cdot (-\lambda r_{2i},r_{2i},-d_{2i-1,h_{2i-1}},b_{2i-1,h_{2i-1}})+ j \cdot \alpha_{r_k}.\\ \end{array}$$}
	Thus, $\alpha_a$ is written as a combination of the following vectors
	\begin{equation}
		\label{Exp-alpha-k-even}
		\alpha_a= \alpha_{n-r_1}+\displaystyle\sum_{i=1}^{\frac{k-2}{2}} h_{2i}\cdot \alpha_{r_{2i}}+ j \cdot \alpha_{r_k},
	\end{equation}
	which means, it is given in terms of vectors associated with binomials in Theorem \ref{Teo-Eq-Surf-C4}. Therefore, $Y^aW^{b_{k,j}}-X^{\lambda a}Z^d \in I_S \setminus {\mathcal G}$.
\end{proof}

\begin{proposition}
	\label{Prop-No-Generators-2}
	Consider $b_0 \in \mathbb{N}$ fixed, with $b_0 \neq b_{k,j}$, and $b_{k,j}$ as in Theorem \ref{Teo-Eq-Surf-C4}, for any $k \in \{1, \ldots, q\}$ and $j \in \{1, \ldots, h_k\}$. Let $d,a \in \mathbb{N}\cup \{0\}$ be any solution of the Diophantine equation $n \cdot d+l \cdot a = m\cdot b_0$, with $l=\pm 1$. Then $$Y^aZ^d-X^{\lambda a}W^{b_0}, \  Y^aW^{b_0}-X^{\lambda a}Z^d \in I_S \setminus {\mathcal G}.$$
\end{proposition}
\begin{proof}
	Our goal is to prove that for $b_0\neq b_{k,j}$, the vectors associated with the binomials $Y^aZ^d-X^{\lambda a}W^{b_0}$ and $Y^aW^{b_0}-X^{\lambda a}Z^d$ are obtained as a combination of vectors associated with the binomials given in Theorem \ref{Teo-Eq-Surf-C4}.
	
	If $b_0 \neq b_{k,j}\in \mathbb{N}$, then $b_0 > h_1=b_{1,h_1}$ since $b_{1,j}=j$ for $j \in \{1,\ldots, h_1\}$ were considered in Theorem \ref{Teo-Eq-Surf-C4}. Moreover, there exist $k \in \{2,\ldots, q\}$ and $j\in \{1,\ldots, h_k\}$ such that $b_{k,j} < b_0 < b_{k,j+1}$, where we define $b_{k,h_k+1}:=b_{k+1,1}$ and $b_{q,h_{q+1}}=\infty$ (see (\ref{Succ-Div-Meth})). If $b_{k,j}<  b_0 < b_{k,j+1}$ for some $k\in \{2, \ldots , q\}$ and $j \in \{1, \ldots, h_k\}$ with $j \neq h_q$, we can write $b_0=b_{k,j}+\theta$ for some $\theta \in \{1,\ldots, b_{k-1,h_{k-1}}-1\}$, since $b_{k,j}=b_{k-1,h_{k-1}}j+b_{k-2,h_{k-2}}$ and $$b_{k,j}+b_{k-1,h_{k-1}} = b_{k-1,h_{k-1}}(j+1)+b_{k-2,h_{k-2}}= b_{k,j+1}.$$ If $b_0 > b_{q,h_q}$, then we write $b_0 = b_{q,h_q}+ \theta$, for some $\theta \in \mathbb{N}$. In order to determine a solution for $n \cdot d + l \cdot a=m \cdot b_0$ with $l=\pm 1$ and $b_0=b_{k,j}+\theta$, note that  
	\begin{equation}
		\label{Sol-Dioph-No-generators}
		\begin{array}{rcl}
			n\cdot d+l\cdot a & = & h_0(b_{k,j}+\theta)n+r_1b_{k,j}+r_1\theta \vspace{0.2cm}\\
			&=&  (h_0b_{k,j}+h_0\theta)n+r_1(b_{k-1,h_{k-1}}j+b_{k-2,h_{k-2}})+r_1\theta .\\
	\end{array}\end{equation}
	
	If $k$ is an odd integer and $l=-1$, by (\ref{Exp-r1b_k,h_k}) and (\ref{Solution-k-even}), we have 
	$$\begin{array}{rcl}
		n\cdot d- a & = &  (h_0b_{k,j}+h_0\theta+c_{k-1,h_{k-1}}j+c_{k-2,h_{k-2}})n+r_{k}j-r_{k-1}+r_1\theta\vspace{0.2cm}\\
		&=&  (d_{k,j}+h_0\theta)n-(r_{k-1}-r_{k}j)+r_1\theta\vspace{0.2cm}\\
		&=&  (d_{k,j}+h_0\theta)n-(r_{k-1}-r_{k}j)-(n-r_1)\theta+n\theta\vspace{0.2cm}\\
		&=&  (d_{k,j}+h_0\theta+\theta)n-(a_{k,j}+(n-r_1)\theta).
	\end{array}$$
	Thus, a solution in $\mathbb{N}$ for this equation is $d=d_{k,j}+(h_0+1)\theta \ $ and $ \ a=a_{k,j}+(n-r_1)\theta.$
	In this case, the vector $\alpha_a$ is written in terms of $\alpha_{a_{k,j}}$ in (\ref{Vector-a_{k,j}}) and $\alpha_{a_{1,1}}=\alpha_{n-r_1}$ in (\ref{Rel-PairVectors-keven}), that is,
	$$\begin{array}{rcl}(-\lambda a,a,-d,b_0) &=&
		(-\lambda a_{k,j},a_{k,j},-d_{k,j},b_{k,j}) + \theta (-\lambda(n-r_1),n-r_1,-(h_0+1),1)\vspace{0.1cm}\\
		&=& \alpha_{a_{k,j}}+\theta \cdot \alpha_{n-r_1}.
	\end{array}$$ 
	
	Similarly to the previous case, if $l=1$ it follows from the proof of item $(i)$ in Proposition \ref{Prop-No-Generators} that
	$$\begin{array}{rcl}
		n\cdot d+a & = &  (d_{k,j}+h_0\theta)n-(r_{k-1}-r_{k}j)+r_1\theta\vspace{0.2cm}\\
		&=&  (d_{k,j}+h_0\theta)n-n+n-(r_{k-1}-r_{k}j)+r_1\theta\vspace{0.2cm}\\
		&=&  (d_{k,j}+h_0\theta-1)n+(n-r_{k-1})+r_{k}j+r_1\theta.
	\end{array}$$
	As consequence, a solution of the previous equation is $d=d_{k,j}+h_0\theta-1=d_{k-1,h_{k-1}}j+d_{k-2,h_{k-2}}-1+h_0\theta$ and $a=(n-r_{k-1})+r_{k}j+r_1\theta$.
	
	\vspace{0.1cm}
	By (\ref{PairVectors-Sol-1}), (\ref{Exp-beta_a-k-odd}), and (\ref{Vector-a_{k,j}}) the vector associated with this solution is written as
	$$\small{\begin{array}{rcl}(-\lambda a,a,d,-b_0) &=&
			(-\lambda (n-r_{k-1}),n-r_{k-1},d_{k-2,h_{k-2}}-1,-b_{k-2,h_{k-2}}) \vspace{0.2cm}\\
			&&+ j(-\lambda r_k,r_k,d_{k-1,h_{k-1}},-b_{k-1,h_{k-1}} )+ \theta (-\lambda r_1,r_1,h_0,-1)\vspace{0.2cm}\\
			&=& \displaystyle\sum_{i=0}^{\frac{k-3}{2}}h_{2i+1}\beta_{r_{2i+1}}+j \cdot \beta_{r_k}+\theta \cdot \beta_{r_1}.
	\end{array}}$$
	
	If $k$ is an even integer and $l=1$, then by (\ref{Exp-r1b_k,h_k}) and (\ref{Solution-k-even}) we can write the equation (\ref{Sol-Dioph-No-generators}) of the form
	$$\begin{array}{rcl}
		n\cdot d+ a & = &  (h_0b_{k,j}+h_0\theta)n+(c_{k-1,h_{k-1}}n-r_{k})j+nc_{k-2,h_{k-2}}+r_{k-1}+r_1\theta\vspace{0.2cm}\\
		&=&  (h_0b_{k,j}+h_0\theta+c_{k-1,h_{k-1}}j+c_{k-2,h_{k-2}})n-r_{k}j+r_{k-1}+r_1\theta.\vspace{0.2cm}\\
		&=&  (d_{k,j}+h_0\theta)n+(r_{k-1}-r_{k}j)+r_1\theta.\\
	\end{array}$$
	Thus, a solution in $\mathbb{N}$ for this equation is $d=d_{k,j}+h_0\theta \ $ and $ \ a=a_{k,j}+r_1\theta.$
	
	In this case, by (\ref{Vector-a_{k,j}}) we get
	$$\begin{array}{rcl}
		\beta_a= (-\lambda a,a,d,-b)&=&
		(-\lambda a_{k,j},a_{k,j},d_{k,j},-b_{k,j})+ \theta \cdot (-\lambda r_1,r_1,h_0,-1)\vspace{0.1cm}\\
		&=& \beta_{a_{k,j}}+ \theta \cdot \beta_{r_1}.
	\end{array}$$
	
	If $l=-1$, similarly to the proof of item $(ii)$ in Proposition \ref{Prop-No-Generators}, we have
	$$\begin{array}{rcl}
		n\cdot d- a & = &  (d_{k,j}+h_0\theta)n+(r_{k-1}-r_{k}j)+r_1\theta\vspace{0.2cm}\\
		&=&  (d_{k,j}+h_0\theta)n+n-(n-(r_{k-1}-r_{k}j))-(n-r_1)\theta+n\theta \vspace{0.2cm}\\
		&=& (d_{k,j}+h_0\theta+1+\theta)n-(n-r_{k-1}+r_{k}j+(n-r_1)\theta).
	\end{array}$$
	In this way, a solution for the previous equation is $d=d_{k,j}+1+(h_0+1)\theta$ and $a=n-r_{k-1}+r_{k}j+(n-r_1)\theta$.
	
	By (\ref{Exp-alpha-k-even}), we write
	$$\small{\begin{array}{rcl}(-\lambda a,a,-d,b_0) &=&
			(-\lambda (n-r_{k-1}),n-r_{k-1},-(d_{k-2,h_{k-2}}+1),b_{k-2,h_{k-2}}) \vspace{0.2cm}\\
			&&+ j(-\lambda r_k,r_k,-d_{k-1,h_{k-1}},b_{k-1,h_{k-1}} ) \vspace{0.2cm}\\
			&&+ \theta (-\lambda (n-r_1),n-r_1,-(h_0+1),1)\vspace{0.2cm}\\
			&=& \alpha_{n-r_1}+\displaystyle\sum_{i=0}^{\frac{k-2}{2}}h_{2i}\alpha_{r_{2i}}+j \cdot \alpha_{r_k}+\theta \cdot \alpha_{n-r_1}.
	\end{array}}$$
	
	In all the cases analysed, we express the vector of a solution (not necessarily the minimal solution) of the Diophantine equation as a combination of vectors associated with the binomials given in Theorem \ref{Teo-Eq-Surf-C4}. The non-minimality of the previous solutions is not a problem, since according to Proposition \ref{Binom-General-Solution}, if one vector is expressed as a combination of vectors associated with the binomials in Theorem \ref{Teo-Eq-Surf-C4}, then  any vector associated with other solution of the same equation (including the minimal solution) can also be expressed in terms of vectors of this theorem.  
\end{proof}

Let us remember that a finitely generated semigroup $S$ induces a homomorphism of groups $\pi_{S} : \mathbb{Z}^p \to \mathbb{Z}^s$, whose kernel determines the ideal $I_S$ of the toric variety (see (\ref{nucleo})). Consider $B$ the set of vectors associated with the binomials in $\mathcal{G}$, that is, $$B=\{(-\lambda n, n, -1,0), (-\lambda r_1,r_1, h_0,-1), u_{k,j}, \ v_{k,j}; \ k=1, \ldots, q, \ j=1,\ldots,h_k\}$$ in which $u_{k,j}=(-\lambda a_{k,j},a_{k,j}, -d_{k,j},b_{k,j} )$ and $v_{k,j}=(-\lambda a_{k,j},a_{k,j}, d_{k,j},-b_{k,j} )$.

\vspace{0.2cm}
A vector $v=(v_1,v_2,v_3,v_4) \in \mathbb{Z}^4$ is called primitive if its components are coprime. In the following result, we show that the vectors in $B$ are primitive, this implies that the corresponding binomials in $\mathcal{G}$ are irreducible in $\mathbb{C}[X,Y,Z,W]$.
\begin{corollary}
	The set $B$ generates $Ker(\pi_S)$ as a $\mathbb{Z}$-module. Moreover, all vectors in $B$ are primitive.
\end{corollary}
\begin{proof}
	It follows from the proofs of Theorem \ref{Teo-Eq-Surf-C4}, Proposition \ref{Prop-No-Generators} and Proposition \ref{Prop-No-Generators-2} that $B$ generates $Ker(\pi_S)$ as a $\mathbb{Z}$-module.
	
	For the second part, note that the first two vectors in $B$ are clearly primitive. We now show that $d_{k,j}$ and $b_{k,j}$ are coprime for all $k,j$. To simplify the notation, we denote $d_{k,h_k}$ and $b_{k,h_k}$ by $d_k$ and $b_k$, respectively. Thus, $d_{k,j}=d_{k-1}\cdot j + d_{k-2}$ and $b_{k,j}=b_{k-1}\cdot j + b_{k-2}$. Recalling the continued fraction expansion $\frac{m}{n}=[h_0;h_1,\ldots,h_q]$ (see (\ref{Succ-Div-Meth})), observe that for each $k$, with $j=h_k$, the convergents of $\frac{m}{n}$ are given by  $\frac{d_0}{b_0}=h_0$, $\ \frac{d_1}{b_1}=\frac{h_0h_1+1}{h_1}=h_0+\frac{1}{h_1}=[h_0;h_1]$, $$\frac{d_2}{b_2}=\frac{(h_0h_1+1)h_2+h_0}{h_1h_2+1}=h_0+\frac{1}{h_1+\frac{1}{h_2}}=[h_0;h_1,h_2],$$ and so on, the $k$-th convergent is $\frac{d_{k}}{b_{k}}=[h_0;h_1,\ldots, h_k].$  By the properties of continued fractions convergents, we have that $d_k\cdot b_{k-1} - d_{k-1} \cdot b_k=(-1)^{k-1}$ which implies that $gcd(d_k,b_k)=1$. Furthermore,  $d_{k,j}$ and $b_{k,j}$ are also coprime, since $gcd(d_{k,j},b_{k,j})$ must divide any linear combination of them. In particular, it divides the relation $$\begin{array}{rcl}
		b_{k-1}d_{k,j}-d_{k-1}b_{k,j}& = & b_{k-1}(d_{k-1} \cdot j+d_{k-2})- d_{k-1}(b_{k-1} \cdot j+b_{k-2})\\
		&=& -(d_{k-1} \cdot b_{k-2}- d_{k-2} \cdot b_{k-1})=(-1)^{k-1}.
	\end{array}$$ Then, $gcd(d_{k,j},b_{k,j})=1$. Therefore $u_{k,j}$ and $v_{k,j}$ are primitive.
\end{proof}

\section{Conflit of interests}

The authors declare that they have no known competing financial interests or personal relationships that could have appeared to influence the work reported in this paper.

\begin{center}
\begin{tabular}{c}
T. M. Dalbelo \\
{\it thaisdalbelo@ufscar.br}\\
Departamento de Matem\'atica, 
Universidade Federal de S\~ao Carlos - UFSCar,
Brazil \vspace{0.2cm}\\
Rodrigues Hernandes, M. E. \\
{\it merhernandes$@$uem.br}\\
Departamento de Matem\'atica, Universidade Estadual de Maringá - UEM,
Brazil \vspace{0.2cm}\\
M. A. S. Ruas\\
maasruas@icmc.usp.br\\
Instituto de Ci\^encias Matem\'aticas e de
Computa\c c\~ao, Universidade de S\~ao Paulo - USP,
Brazil
\end{tabular}
\vspace{0.5cm}

\end{center}

\end{document}